\documentclass[12pt]{article}
\usepackage[cp1251]{inputenc}
\usepackage[english]{babel}
\usepackage{amsmath}
\usepackage{mathrsfs}
\usepackage{amsfonts,amssymb}
\usepackage{xkeyval,ifthen}
\usepackage{delarray}
\newtheorem{theorem}{Theorem}
\newtheorem{lemma}{Lemma}
\begin{document}
\title{\textbf{Steady-state oscillations
      of a half-cylinder with a pseudo isotropic structure}}
\author{\textbf{Mario Argueta Garcia}\\
        Academia de matematicas SLT\\
        Colegio de Ciencia y Tecnologia\\
        Universidad Autonoma de la Ciudad de Mexico\\
        C.P. 09790, Mexico D. F.\\
        mario.argueta@uacm.edu.mx}
\date{}
\maketitle
\begin{abstract}
    In this paper we investigate the spectral properties of certain
    self-adjoint quadratic pencils, occurring at the separation of
    the variables in the problem of the oscillation of a
    semi-infinite cylinder with an isotropic elastic structure, when the Lame
    constants are $\mu>0$ and $\lambda=0$ or $\mu>0$ and $\lambda=\infty$.
\end{abstract}

\textbf{\normalsize  1.~Spectral problems
    $\boldsymbol{\mathcal{L}_\omega(\alpha)}$ and
    $\boldsymbol{\mathcal{L}_\omega^0(\alpha)}$ in
    the case of an isotropic structure}

\bigskip
\textbf{1.1~Statement of model problems}\\

    We consider a system of equation of small oscillation of an
    elastic medium in the semi-infinite cylinder $\Omega$, where
    $\Omega=\textbf{R}_{x_1}^+\times\mathcal{D}\subset\textbf{R}^3$;
    where $\mathcal{D}\subset\textbf{R}^2=\{(0,x_2,x_3)\}$~---
    is a bounded domain with a smooth boundary $\partial\mathcal{D}$,

\begin{equation} \label{n1}
   \sum_{k=1}^3 \frac{\partial\sigma_{ik}(u)}{\partial
   x_k}=\rho\frac{\partial^2u_i}{\partial t^2},\qquad i=1,2,3,
\end{equation}\\
   where $u=u(t,x_1,x_2,x_3)=(u_1,u_2,u_3)$~--- is the displacement vector,
   $\sigma=(\sigma_{ik})^3_{i,k}$ --- is the stress tensor,
   $\rho=\rho(x_2,x_3)$~--- is the density of the medium, which is independent of $x_1$
   and \,$0<m<\rho(x)<M$. On the lateral surface
   $\Gamma=\partial\mathcal{D}\times\textbf{R}^+_{x_1}$, of the semi-infinite cylinder
   $\Omega$, the following boundary condition holds:
\begin{align}
   & \left. {u} \right|_\Gamma=0 \label{n2},\\
   \intertext{or}
   & \left. {\sigma(u)\cdot n} \right|_\Gamma=0 \label{n3},
\end{align}
   where $n=n(x_2,x_3)=(0,n_2,n_3),\,n_2^2+n_3^2=1$~--- is the exterior normal
   to $\Gamma$. The first condition corresponds to a fixed boundary,
   and the second, to a free boundary, see \cite{JLLyEM72}.
   Separating the variables
   $u(t,x_1,x_2,x_3)=e^{i\omega t}u(x_1,x_2,x_3),\,\omega\in\textbf{R}$,
   and substituting this expression in (1)--(2), we obtain equation
   of the stead-state oscillation of the semi-infinite cylinder $\Omega$,
   see \cite{TANSAA77}:
\begin{equation}\label{n4}
   \sum_{k=1}^3 \frac{\partial\sigma_{ik}(u)}{\partial
   x_k}+\rho\omega^2u_i=0,\qquad i=1,2,3,
\end{equation}
   whit the condition on the lateral surface:
\begin{align}
   & \left. {u} \right|_\Gamma=0 \label{n5},\\
\intertext{or}
   & \left. {\sigma(u)\cdot n} \right|_\Gamma=0 \label{n6}.
\end{align}
\par\medskip
   We also assume that, in addition to the conditions (\ref{n5}) and (\ref{n6}),
   on the base of semi-cylinder $\Omega$ there are given definite initial data.
   Very often one gives condition of the form:
\begin{align}
   & \left. {u} \right|_\mathcal{D}=\varphi \label{n7},\\
\intertext{or}
   & \left. {\sigma(u)\cdot n} \right|_\mathcal{D}=\psi \label{n8},
\end{align}\\
   where $n=(-1,0,0)$~--- is a normal to $\mathcal{D}$, while  $\varphi$ and $\psi$~---
   are functions from a definite class of smoothness. At the infinity, we set the condition:
\begin{align}
   u(\infty,x_2,x_3)=0.\label{n9}
\end{align}
   While examining the elastic properties of a medium , we deal with
   the relation between the components of the corresponding stress
   and strain tensors. The components of the tensor stress can be
   obtained by the differentiating the free energy (elastic energy)
   $W$ of the system, with respect to the strain tensor components;
   i.e.,
\begin{equation}\label{n10}
   \sigma_{ik}(u)=\frac{\partial W}{\partial e_{ik}}\, , \qquad
   i,k=1,2,3,
\end{equation}

   where

\begin{equation}\label{n11}
   e_{ik}=\frac{1}{2}\left(\frac{\partial u_i}{\partial
   x_k}+\frac{\partial u_k}{\partial x_i}\right). \qquad
   i,k=1,2,3,
\end{equation}
\par\medskip
   In the case of isotropic elastic medium (isotropic structure),
   the free energy is given by:
\begin{equation}\label{n12}
   W=\frac{1}{2}\lambda(e_{11}+e_{22}+e_{33})^2+\mu(e^2_{11}+
   e^2_{22}+e^2_{33}+e^2_{12}+e^2_{13}+e^2_{23}),
\smallskip
\end{equation}
  (where $\lambda$ and $\mu$ are Lame constants of the isotropic structure, see \cite{LLDyLEM87,BLMyGVV82}).

  We have some spacial cases for the isotropic structure, when Lame
  constants are $\mu>0$ and $\lambda=0$ or $\mu>0$ and $\lambda=\infty$
  (the case, when $\lambda >0$ and $\mu >0$, see \cite{KAGyOMB81}).

  We examine the first case for the isotropic structure, when the Lame
  constants are $\mu>0$ and $\lambda=0$, see \cite{SIS83,ALP,WHM40}.
  Then we arrive at the following free energy $W$ of the system:
\begin{equation}\label{n13}
   W=\mu(e^2_{11}+e^2_{22}+e^2_{33}+e^2_{12}+e^2_{13}+e^2_{23}).
\end{equation}

   We find the relations (10), when $\lambda =0$. Then
\[
\begin{matrix}
    \sigma_{11}=2\mu e_{11},&\qquad \sigma_{12}=2\mu e_{12},&\qquad
    \sigma_{13}=2\mu e_{13},\\
    \sigma_{21}=2\mu e_{21},&\qquad \sigma_{22}=2\mu e_{22},&\qquad
    \sigma_{23}=2\mu e_{23},\\
    \sigma_{31}=2\mu e_{31},&\qquad \sigma_{32}=2\mu e_{32},&\qquad
    \sigma_{33}=2\mu e_{33}.
\end{matrix}
\]
  Now we can find the relations (\ref{n4}) in this case. Then
\[
    \frac{\partial \sigma_{11}}{\partial
    x_1}=2\mu\frac{\partial^2u_1}{\partial x^2_1},\quad \frac{\partial
    \sigma_{12}}{\partial x_2}=\mu\frac{\partial^2u_1}{\partial x^2_2}
    +\mu \frac{\partial^2u_2}{\partial x_2 \partial x_1},\quad
    \frac{\partial \sigma_{13}}{\partial x_3}=\mu\frac{\partial^2u_1}
    {\partial x^2_3}+\mu \frac{\partial^2u_3}{\partial x_3 \partial x_1},
\]

\[
    \frac{\partial \sigma_{21}}{\partial x_1}=\mu\frac{\partial^2u_2}
    {\partial x^2_1}+\mu \frac{\partial^2u_1}{\partial x_2 \partial
    x_1},\quad \frac{\partial \sigma_{22}}{\partial x_2}=2\mu
    \frac{\partial^2u_2}{\partial x^2_2},\quad \frac{\partial \sigma_{23}}
    {\partial x_3}=\mu\frac{\partial^2u_2} {\partial x^2_3}+\mu
    \frac{\partial^2u_3}{\partial x_2 \partial x_3},
\]

\[
    \frac{\partial \sigma_{31}}{\partial x_1}=\mu\frac{\partial^2u_3}
    {\partial x^2_1}+\mu \frac{\partial^2u_1}{\partial x_3 \partial
    x_1},\quad \frac{\partial \sigma_{32}}{\partial x_2}=\mu
    \frac{\partial^2u_3}{\partial x^2_2}+\mu \frac{\partial^2u_2}
    {\partial x_2 \partial x_3},\quad \frac{\partial \sigma_{33}}{\partial x_3}
    =2\mu\frac{\partial^2u_3}{\partial x^2_3}.
\]
  In that way
\[
    \frac{\partial \sigma_{11}}{\partial x_1}+\frac{\partial
    \sigma_{12}}{\partial x_2}+\frac{\partial \sigma_{13}}{\partial
    x_3}+\rho w^2u_1=2\mu\frac{\partial^2u_1}{\partial x^2_1}+
    \mu \frac{\partial^2u_2}{\partial x_2 \partial x_1}+
    \mu\frac{\partial^2u_3}{\partial x_3 \partial x_1}+
    \mu\frac{\partial^2u_1}{\partial x^2_2}+
    \mu\frac{\partial^2u_1}{\partial x^2_3}+\rho w^2u_1=0,
\]

\[
    \frac{\partial \sigma_{21}}{\partial x_1}+\frac{\partial
    \sigma_{22}}{\partial x_2}+\frac{\partial \sigma_{23}}{\partial
    x_3}+\rho w^2u_2=\mu\frac{\partial^2u_2}{\partial x^2_1}+
    \mu \frac{\partial^2u_1}{\partial x_2 \partial x_1}+
    2\mu\frac{\partial^2u_2}{\partial x^2_2}+
    \mu\frac{\partial^2u_2}{\partial x^2_3}+
    \mu\frac{\partial^2u_3}{\partial x_2 \partial x_3}+\rho w^2u_2=0,
\]

\[
    \frac{\partial \sigma_{31}}{\partial x_1}+\frac{\partial
    \sigma_{32}}{\partial x_2}+\frac{\partial \sigma_{33}}{\partial
    x_3}+\rho w^2u_3=\mu\frac{\partial^2u_3}{\partial x^2_1}+
    \mu \frac{\partial^2u_1}{\partial x_3 \partial x_1}+ \mu
    \frac{\partial^2u_2}{\partial x_2 \partial x_3}+\mu\frac{\partial^2u_3}
    {\partial x^2_2}+ 2\mu\frac{\partial^2u_3}{\partial x^2_3}+
    \rho w^2u_3=0.
\]
\par\medskip
   In what follows, it is convenient to set $x_1=y$. After
   substitution of these expressions into equation (\ref{n4}) it takes the form:
\begin{equation}\label{n14}
   -\mathcal{A}\frac{d^2u}{dy^2}+i\mathcal{B}\frac{du}{dy}+
   (\mathcal{C}-\omega^2\mathcal{R})u=0, \quad (y=x_1>0),
\end{equation}
   where $u=(u_1,u_2,u_3)$ and \\
\[
\mathcal{A}=\mu
\begin{array} \lgroup{ccc}\rgroup
   {\! \! \!} 2     &   0  &  0 {\! \! \!}  \\
   {\! \! \!} 0     &   1  &  0  {\! \! \!}  \\
   {\! \! \!} 0     &   0  &  1 {\! \! \!}
\end{array}\!,
\quad \mathcal{B}=i\mu
\begin{array} \lgroup{ccc}\rgroup
   {\! \! \!} 0       &     D_2     &    D_3 {\! \! \!} \\
   {\! \! \!} D_2     &      0      &      0 {\! \! \!}   \\
   {\! \! \!} D_3     &      0      &      0 {\! \! \!}
\end{array}\!,
\quad \mathcal{R}=
\begin{array} \lgroup{ccc}\rgroup
    {\! \! \!}    \rho  &    0    &    0  {\! \! \!}   \\
    {\! \! \!}       0  &  \rho   &    0  {\! \! \!}   \\
    {\! \! \!}       0  &    0    &  \rho {\! \! \!}
\end{array}\!,
\]
\bigskip
\[
\mathcal{C}=-\mu
\begin{array} \lgroup{ccc}\rgroup
   {\! \! \!} D^2_2+ D^2_3   &           0       &          0            \\
         0                   &     2D^2_2+D^2_3  &       D_2D_3            \\
         0                   &        D_2D_3     &    D^2_2+2D^2_3 {\! \! \!}
\end{array}\!,
\]
\medskip
\[
   D_k=\frac{\partial}{\partial x_k},\quad
   \text{where}\quad k=2,3.
\]
  Now we can find the form of equation (\ref{n6}). Then
\[
   \sigma(u) \cdot n=(\sigma_{12}n_2+\sigma_{13}n_3,\,\sigma_{22}n_2+
   \sigma_{23}n_3,\,\sigma_{32}n_2+\sigma_{33}n_3)=
\]
\[
   =(\mu n_2\frac{\partial u_1}{\partial x_2}+\mu n_3\frac{\partial
   u_1}{\partial x_3}+\mu n_2\frac{\partial u_2}{\partial x_1}+
   \mu n_3\frac{\partial u_3}{\partial x_1}\,,\,2\mu n_2\frac{\partial u_2}
   {\partial x_2}+\mu n_3\frac{\partial u_2}{\partial x_3}+\mu
   n_3\frac{\partial u_3}{\partial x_2}\,,
\]

\[
   \mu n_2\frac{\partial u_2}{\partial x_3}+\mu n_2\frac{\partial
   u_3}{\partial x_2}+2\mu n_3\frac{\partial u_3}{\partial x_3}).
\]\\
   Equation (\ref{n6}) can be written as:
\begin{equation} \label{n15}
   (\mathcal{M}u-i\mathcal{N}\frac{du}{dy})_\Gamma=0,
\end{equation}
  where
\[
\mathcal{M}=\mu
\begin{array} \lgroup{ccc}\rgroup
  {\! \! \!} n_2D_2+n_3D_3    &           0        &       0             \\
                   0          &   2n_2D_2+n_3D_3   &    n_3D_2            \\
                   0          &        n_2D_3      & n_2D_2+2n_3D_3 {\! \! \!}
\end{array}\!,
\]
\medskip
\[
\mathcal{N}=\mu
\begin{array} \lgroup{ccc}\rgroup
  {\! \! \!} 0 &    n_2    &    n_3 {\! \! \!} \\
  {\! \! \!} 0 &     0     &   0    {\! \! \!}   \\
  {\! \! \!} 0 &     0     &   0    {\! \! \!}
\end{array}\!.
\]

\medskip
   Functions $u$ satisfying equations (\ref{n14})--(\ref{n15})
   and having the representation:
\begin{equation} \label{n16}
   u=e^{i\alpha y}v(x_2,x_3),
\end{equation}
   are called proper oscillations of semi-cylinder \, $\Omega$ \,
   with a free boundary $\Gamma$. The number $\alpha$, is called a wave number or
   eigenvalue, and the function $v$~--- is called the amplitude or eigenfunction.
   Substituting into equations (\ref{n14})--(\ref{n15}) a function
   of the form (\ref{n16}), we find that the problem of steady-state oscillations
   of the semi-infinite cylinder with an isotropic structure and whit a free
   boundary, when the Lame constants are $\mu>0$ and $\lambda=0$,
   there corresponds the following spectral problem (eigenvalue problem)
   with the parameter $\alpha$:

\begin{equation} \label{n17}
   \mathcal{C}_\omega (\alpha)v \equiv(\alpha^2\mathcal{A}-
   \alpha\mathcal{B}+\mathcal{C}-\omega^2\mathcal{R})v=0,
\end{equation}

\begin{equation} \label{n18}
   \left. {\mathcal{U}(\alpha)v} \right|_{\partial\mathcal{D}}\equiv
   \left. {(\mathcal{M}+i\alpha\mathcal{N})v} \right|_{\partial
   \mathcal{D}}=0.
\end{equation}\\
   The problem (\ref{n17})--(\ref{n18}) we denoted by $\mathcal{L}_\omega
   (\alpha)v$. Therefore, the equality  $\mathcal{L}_\omega(\alpha)v=f$ means that
   $\mathcal{C}_\omega (\alpha)v=f$ and $\left. {\mathcal{U}(\alpha)v}\right|_{\partial\mathcal{D}}=0$.
   A nonzero vector function $y_0(x_2,x_3)=(y_{01},y_{02},y_{03})$
   is called a eigenvector of the spectral problem $\mathcal{L}_\omega(\alpha)$,
   corresponding to eigenvalue $\alpha_0$, if $\mathcal{L}_\omega(\alpha_0)y_0=0$. The system
   $y_0,\dots, y_p$ is called the chain of eigen-and associated vectors, corresponding
   to the eigenvalue $ \alpha_0$, if
\[
   \mathcal{L}_\omega(\alpha_0)y_l+\frac{\partial}{\partial\alpha}\mathcal{L}_\omega(\alpha_0)
   y_{l-1}+\frac{1}{2}\frac{\partial^2}{\partial\alpha^2}\mathcal{L}_\omega(\alpha_0)y_{l-2}=0,
   \, y_{-1}=y_{-2}=0,\, l=0,\dots, p,
\]
  in a more detailed manner,
\begin{equation} \label{n19}
   \mathcal{C}_\omega(\alpha_0)y_l+\frac{\partial}{\partial\alpha}\mathcal{C}_\omega(\alpha_0)
   y_{l-1}+\frac{1}{2}\frac{\partial^2}{\partial\alpha^2}\mathcal{C}_\omega(\alpha_0)y_{l-2}=0
   \quad \text{in} \quad \mathcal{D},
\end{equation}
\begin{equation} \label{n20}
   \mathcal{U}(\alpha_0)y_l+\frac{\partial}{\partial\alpha}\mathcal{U}(\alpha_0)y_{l-1}=0
   \quad \text{on} \quad \partial\mathcal{D}.
\end{equation}\\
   The elements of a system of eigen-and associated vectors are
   sometimes called roots vectors.

   If $\alpha_0$~--- is an eigenvalue of the problem $\mathcal{L}_\omega(\alpha)$ and $y_0,\dots,y_p$~---
   is some corresponding chain of eigen-and associated vectors, then
   the solution of the problem (\ref{n4}) and (\ref{n6}) of the form:

\[
   u_{\text{e}}(x_1,x_2,x_3)=e^{i\alpha_0x_1}\left(\frac{(ix_1)^p}{p\,!}y_0+\frac{(ix_1)^{p-1}}{(p-1)!}y_1+
   \cdots +y_p \right),
\]\\
   are called elementary solutions and the solutions $u(t,x_1,x_2,x_3)=e^{i\omega t}u$,
   of the problem (\ref{n1}) and (\ref{n3}) are called normal oscillations (or normal waves).

   In the case of the fixed boundary, condition (\ref{n8}) is replaced by:
\begin{align}
   \left. {v} \right|_\Gamma=0 \label{n21},
\end{align}
  we arrive at a spectral problem with the parameter $\alpha$:
\begin{equation} \label{n22}
   \mathcal{C}_\omega (\alpha)v \equiv(\alpha^2\mathcal{A}-
   \alpha\mathcal{B}+\mathcal{C}-\omega^2\mathcal{R})v=0,
\end{equation}
\begin{align}
   \left. {v} \right|_\Gamma=0 \label{n23},
\end{align}
  The spectral problem (\ref{n22})--(\ref{n23}) will be denoted by $\mathcal{L}^0_\omega(\alpha)$.
  The above introduced concepts can be defined also for this problems.

   In this paper, we examine the spectral properties of the problems $\mathcal{L}_\omega(\alpha)$
   and $\mathcal{L}^0_\omega(\alpha)$ for the isotropic structure, when the Lame constants are
   $\mu>0$ and $\lambda=0$. The spectral properties of the problems $\mathcal{L}_\omega(\alpha)$ and
   $\mathcal{L}^0_\omega(\alpha)$ for the isotropic system, when Lame constants are $\mu>0$ and
   $\lambda>0$, were examined by A.~G.~Kostyuchenko and M.~B.~Orazov in the paper
   \cite{KAGyOMB81}.\\

\bigskip

\textbf{\normalsize  2.~Reduction of the spectral problems
   $\boldsymbol{\mathcal{L}_\omega(\alpha)}$ and
   $\boldsymbol{\mathcal{L}_\omega^0(\alpha)}$
   to the pencils $\boldsymbol{L_\omega(\alpha)}$ and $\boldsymbol{L_\omega^0(\alpha)}$}

\bigskip

\textbf{2.1~Reduction of the spectral problem
   $\boldsymbol{\mathcal{L}_\omega(\alpha)}$ to the pencil $\boldsymbol{L_\omega(\alpha)}$}.\\

   By $W^1_2(\mathcal{D})$ we denoted the Sobolev space of the vector functions:

\[
   v(x_2,x_3)=(v_1,v_2,v_3),
\]
   having square summable first generalized derivatives, with the norm:
\[
  \|v\|^2_1=\int\limits_\mathcal{D}\left(|\nabla v_1|^2+|\nabla v_2|^2+|\nabla v_3|^2\right)\,dx
  +\int\limits_\mathcal{D}|v|^2\,dx.
\]
  For the operator $\mathcal{C}$ and by a simple verification we can see that the following
  Green identity holds:
\begin{equation} \label{n24}
  \int\limits_\mathcal{D}\mathcal{C}v\overline {g\mathstrut}\,dx=\int\limits_\mathcal{D}\mathcal{E}_0(v,g)\,dx-
  \int\limits_{\partial\mathcal{D}}\mathcal{M}v\overline {g\mathstrut}\,ds,
\end{equation}\\
  where $v\in C^2(\overline{\mathcal{D}\mathstrut})$; $g\in C^1(\overline{\mathcal{D}\mathstrut})$,
  and for $\mathcal{E}_0(v,g)$, when $\mu>0$ and  $\lambda=0$, we find the next equality:\\
\begin{equation} \label{n25}
  \mathcal{E}_0(v,g)=\mu(D_2v_1D_2\overline{g\mathstrut}_1+D_3v_1D_3\overline{g\mathstrut}_1)+
  \frac{1}{2}\mu\sum_{i,k\geq 2}(D_iv_k+D_kv_i)(D_i\overline{g\mathstrut}_k+
  D_k\overline{g\mathstrut}_i)
\end{equation}\\
\[
  \text{where} \:\: D_k=\frac{\partial}{\partial x_k},\: k=2,3.
\]
  we denote

\[
  E_0(v,g)=\int\limits_\mathcal{D}\mathcal{E}_0(v,g)\,dx+\int\limits_\mathcal{D}v\overline{g\mathstrut}\,dx.
\]
  Using the second Korn's inequality, see. \cite{FG74}, we obtain
\begin{align}
  \|v\|^2_1\leq\beta E_0(v,v) \,\quad \text{for} \quad v\in W^1_2(\mathcal{D}), \quad \beta >0. \label{n26}
\end{align}
  Since, for some $\gamma>0$ and for $v\in W^1_2(\mathcal{D})$ we obtain the inequality:
\begin{align}
  \gamma E_0(v,v)\leq \|v\|^2_1, \label{n27}
\end{align}
  from here it follows, that on the space $W^1_2(\mathcal{D})$ the
  norms $E_0(v,v)$ and $\|v\|^2_1$ are equivalent. In addition, from Sovolev's
  imbedding theorems, see \cite{MVP76}, and from relation (\ref{n26})
  for $v\in W^1_2(\mathcal{D})$, we have:
\begin{align} \label{n28}
  \|v\|^2\leq a^2E_0(v,v); \quad \|v\|^2_0\leq b^2E_0(v,v), \quad
  (\|v\|^2_0=\int\limits_{\partial\mathcal{D}}|v|^2\,ds).
\end{align}
  Using the relations (\ref{n24})--(\ref{n28}), the reduction of the spectral problem
  $\mathcal{L}_\omega(\alpha)$, which is a unbounded operator for each $\alpha\in \textbf{C}$,
  to bounded self-adjoint quadratic pencil $L_\omega (\alpha)$, following a
  line of reasoning similar to that used for the isotropic case,
  when the Lame constants are $\mu>0$  and  $\lambda>0$, see \cite{KAGyOMB81}, which
  yields the following equality for the resolvent:

\[
  \mathcal{L}^{-1}_\omega(\alpha)=P^{\frac{1}{2}}L^{-1}_\omega(\alpha)
  P^{\frac{1}{2}},
\medskip
\]
  where the operator \,$P\in\mathfrak{G}_q,\,q>\frac{3}{2} \ \,\text{and} \
  P>0$, where $\mathfrak{G}_q$~--- is the class of completely continuous operators
  for which the series $\sum s_j^p(A)$ converges, where
  $s_j(A)$~---are the $s$-numbers of the operator $A$.\\

\bigskip

\textbf{2.2~Reduction of the spectral problem
  $\boldsymbol{\mathcal{L}^0_\omega(\alpha)}$ to the pencil $\boldsymbol{L^0_\omega(\alpha)}$}.\\

  By space $\overset{0}{W}{^1_2(\mathcal{D})}$~--- we denote the closure of the set
  of all smooth finite functions in $\mathcal{D}$,
\smallskip
  in the norm of $W^1_2(\mathcal{D})$. We have for the operator $\mathcal{C}$,
  for all $v\in C^2(\overline{\mathcal{D}\mathstrut})$;
  $\left.{v}\right|_{\partial\mathcal{D}}=0$ and for all $g\in C^1(\overline{\mathcal{D}\mathstrut})$;
  $\left. {g} \right|_{\partial\mathcal{D}}=0$, from (\ref{n17}) there follows that

\begin{equation}\label{n29}
  (\mathcal{C}v,g)=\int\limits_{\mathcal{D}}\mathcal{E}_0(v,g)\,dx,
\end{equation}
  where $\mathcal{E}_0$ is defined in (\ref{n25}). Using the first Korn's inequality,
  see \cite{FG74}, we obtain to $v\in\overset{0}{W}{^1_2(\mathcal{D})}$ the next equality:

\[
  \gamma\int\limits_{\mathcal{D}}\mathcal{E}_0(v,v)\,dx\,dx\leq\|v\|^2_1\leq\beta\int\limits_{\mathcal{D}}\mathcal{E}_0(v,v)\,dx,
\]
  which mean that on the space \, $\overset{0}{W}{^1_2(\mathcal{D})}$\, the norms $\|v\|^2_1$
  \, and \, $\int\limits_{\mathcal{D}}\mathcal{E}_0(v,v)\,dx$ are equivalent.

  The reduction of the spectral problem $\mathcal{L}^0_\omega(\alpha)$, which is a unbounded operator for each
  $\alpha\in \textbf{C}$, to bounded self-adjoint quadratic pencil $L^0_\omega (\alpha)$,
  following a line of reasoning similar to that used for the isotropic case,
  when the Lame constants are $\mu>0$  and  $\lambda>0$, see \cite{KAGyOMB81}, which
  yields the following equality for the resolvent:
\[
  \mathcal{L}^0_\omega(\alpha)^{-1}=P^{\frac{1}{2}}_0\,L^0_\omega(\alpha)^{-1}
  P^{\frac{1}{2}}_0,
\]
  where the operator \,$P_0\in\mathfrak{G}_q,\,q>\frac{3}{2} \ \,\text{and} \
  P_0>0$.\\

\bigskip

\textbf{3.~Certain properties of the pencils
  $\boldsymbol{L_\omega(\alpha)}$ and $\boldsymbol{L^0_\omega(\alpha)}$}

\bigskip

  Here we formulate a series of statements, regarding the localization of the spectra of quadratics pencils
  $L_\omega(\alpha)$ and $L^0_\omega(\alpha)$ and the estimate of the resolvents for the spectral problems
  $\mathcal{L}_\omega(\alpha)$ and $\mathcal{L}^0_\omega(\alpha)$.
  One shows that for $\omega=0$ (the static case), the pencil $L_0(\alpha)$
  has only one real point of the spectra $\alpha=0$ (the remaining are complex) while the pencil
  $L^0_0(\alpha)$ is weakly damped, see \cite
  {KAGyOMB81,SAAySAV91,KMGyLGK65,KAGySAA83,KAGyOMB75,SAA89}.\\

\bigskip

\textbf{3.1~Static case for the spectral problems
  $\boldsymbol{\mathcal{L}_\omega(\alpha)}$ and $\boldsymbol{\mathcal{L}^0_\omega(\alpha)}$}

\bigskip

  We examine the static case, when $\omega=0$. We have the domain $\Omega_{\varepsilon,N}\smallskip$:
  \\
  $\Omega_{\varepsilon,N}=\{\alpha:|\arg\alpha|<\frac{\pi}{2}-\varepsilon,|\arg\alpha|>\frac{\pi}{2}+\varepsilon,
  -\pi<\arg\alpha\leq\pi,|\alpha|>N\}$, where \,$0<\varepsilon\leq\frac{\pi}{2}$.
\begin{lemma} \label{L1}
  For the isotropic structure, when the Lame constants are $\mu>0$ and $\lambda=0$,
  for any number $\varepsilon$, $0<\varepsilon<\frac{\pi}{2}$, one
  can indicate a number $N$, such that the spectral problems
  $\mathcal{L}_\omega(\alpha)$ and $\mathcal{L}^0_\omega(\alpha)$, are regular elliptic boundary problems
  with the respect to the parameter $\alpha$, in the sense of the
  definition of \cite {AMCyVMI64}, in the domain $\Omega_{\varepsilon,N}$.
\end{lemma}
\textbf{Proof.}\: the problem
  $\mathcal{L}^0_\omega(\alpha)$~--- is a Dirichlet problem for the equation $\mathcal{C}_\omega(\alpha)v=f$
  in the domain $\mathcal{D}$. therefore, for the problem
  $\mathcal{L}^0_\omega(\alpha)$ it is sufficient to verify

\smallskip

  Condition I.
  $DetC_0(\xi,\alpha)\neq 0$ for $\alpha\in\Omega_{\varepsilon,N}$, where $C_0(\xi,\alpha)$ denotes the matrix obtained
  from $\mathcal{C}_0(\alpha)$ of (\ref{n17}), by replacing  $iD_k$ by $\xi_k$,
  where $k=2,3$ and
  $(\xi_2,\xi_3)\in \textbf{R}^2$.

  We obtain
\[
  DetC_0(\xi,\alpha)=\mu^3\left(2\alpha^6+6\alpha^4|\xi|^2+6\alpha^2|\xi|^4+2|\xi|^6\right)=
  2\mu^3\left(\alpha^2+|\xi|^2\right)^3>0,
\]
  where $|\xi|^2=\xi^2_2+\xi^2_3$. From here and where the Lame constant
  $\mu>0$, we obtain at once that condition I holds. Thus the lemma
  is proved for the spectral problem $\mathcal{L}^0_\omega(\alpha)$.

  In order to prove the lemma for the spectral problem $\mathcal{L}_\omega(\alpha)$,
  one has to show, in addition to condition I, also:

  Condition II.
  At each point $(x_2,x_3)\in\partial\mathcal{D}$ the Shapiro"--~Lopatinskii condition holds
  from the spectral problem with the parameter, see \cite{AMCyVMI64}, for
  $\alpha\in\Omega_{\varepsilon,N}$.

  We shall not dwell on the verification of this condition. We formulate only this condition,
  for example, at the point $x_0=(0,0)$, assuming that $x_0\in\partial\mathcal{D}$,
  and that the axis $Ox_3$ is tangent to the boundary $\partial\mathcal{D}$
  at the point $x_0$, while the normal vector $n_0=(1,0)$ is directed along the $Ox_2$ axis.
  Then, condition II for the spectral problem $\mathcal{L}_\omega(\alpha)$ is formulated in the
  following manner. The problem
\[
  \mathcal{C}_0(\alpha,-iD_2,\xi_3)v(y)=0 \quad (y>0),
\]
\[
  \left.{[\mathcal{M}(-iD_2,\xi_3)+i\alpha\mathcal{N}(n_0)]v(y)}\right|_{\,y=0}=h=
\begin{array}\lgroup {c} \rgroup
  {\! \! \!} h_1 {\! \! \!} \\
  {\! \! \!} h_2 {\! \! \!} \\
  {\! \! \!} h_3 {\! \! \!}
\end{array},\,
  h_i\in\textbf{C},
\]
  for $|\xi_3|+|\alpha|\neq 0$, $\alpha\in\Omega_{\varepsilon,N}$ and
  for any $h$ has only one solution in the class $\mathfrak{M}$ of
  solutions, tending to zero together whit the derivatives as $y\rightarrow +\infty$.
  QED.\\

  According to the results of \cite{AMCyVMI64} and lemma~\ref{L1}, there follows that in the domain
  $\Omega_{\varepsilon,N}$ for any $f\in L_2(\mathcal{D})$ one has the estimate:
\begin{align} \label{n30}
  \|\mathcal{L}^{-1}_\omega(\alpha)f\|_2+|\alpha|\|\mathcal{L}^{-1}_\omega(\alpha)f\|_1+
  |\alpha|^2 \|\mathcal{L}^{-1}_\omega(\alpha)f\|\leq r\|f\|,
\end{align}
\begin{align} \label{n31}
  \|\mathcal{L}^0_\omega(\alpha)^{-1}f\|_2+|\alpha|\|\mathcal{L}^0_\omega(\alpha)^{-1}f\|_1+
  |\alpha|^2 \|\mathcal{L}^0_\omega(\alpha)^{-1}f\|\leq r\|f\|.
\end{align}
  At, since
  $\mathcal{L}^{-1}_\omega(\alpha)=P^{\frac{1}{2}}L^{-1}_\omega(\alpha)P^{\frac{1}{2}}$
  and, therefore, taking into account that for any vector $g\in W^1_2(\mathcal{D})$ the norm
  $\|g\|_1\cong\|P^{-\frac{1}{2}}g\|$, for estimate (\ref{n30}), we obtain
\[
  |\overline{\alpha}|\|\mathcal{L}^{-1}_\omega (\overline\alpha)f \|_1 \cong
  |\overline{\alpha}|\|L^{-1}_\omega(\overline\alpha)P^{\frac{1}{2}}f\|\leq r_1\|f\|.
\]
  From here we obtain the relation
  $\|L^{-1}_\omega(\overline{\alpha})P^{\frac{1}{2}}\|\leq
  r_1|\overline{\alpha}|^{-1}$, and thus,
\[
  \|P^{\frac{1}{2}}L^{-1}_\omega(\alpha)\|\leq r_1|\alpha|^{-1},\:
  \text{for}\: \alpha\in\Omega_{\varepsilon,N}.
\]
  Similary, for the spectral problem $\mathcal{L}^0_\omega(\alpha)$ from estimate (\ref{n31})
  and the equality for the resolvent \,$\mathcal{L}^0_\omega(\alpha)^{-1}=P^{\frac{1}{2}}_0\,L^0_\omega(\alpha)^{-1}
  P^{\frac{1}{2}}_0$\, we obtain that
\[
  \|P^{\frac{1}{2}}_0L^0_\omega(\alpha)^{-1}\|\leq r_2|\alpha|^{-1},\:
  \text{for}\: \alpha\in\Omega_{\varepsilon,N}.
\]

\begin{lemma} \label{L2}
  Let $\omega=0$. Then the quadratic pencil $L^0_0(\alpha)$ is uniformly weakly
  damped, i.e., one has the estimate:
  \begin{equation} \label{n32}
  (B_1\zeta,\zeta)^2\leq 4\tau^2(A_1\zeta,\zeta)(\zeta,\zeta),\:\zeta\in L_2(\mathcal{D}),
  \end{equation}
  with some constant $\tau$, where $\frac{1}{2}\leq \tau^2<1$, and where
  $B_1=-P^{\frac{1}{2}}_0\mathcal{B}P^{\frac{1}{2}}_0$ \,and\,
  $A_1=P^{\frac{1}{2}}_0\mathcal{A}P^{\frac{1}{2}}_0$. The pencil
  $L_0(\alpha)$ has on the real axis only one spectral point
  $\alpha=0$, to which there correspond four linearly independent eigenvectors
  of this pencil.
\end{lemma}

\bigskip
\textbf{Proof.}~ At the first, we examine the quadratic pencil
  $L^0_0(\alpha)$ and we will proof the relation (\ref{n32}). Let
  $v=P^{\frac{1}{2}}_0\zeta$, then we have the equalities:
\[
  (B_1\zeta,\zeta)=-(\mathcal{B}v,v)=-2\mu\mathop{\text{Re}}(iD_2v_2+iD_3v_3,v_1),
\]
\[
  (\zeta,\zeta)=(P^{-\frac{1}{2}}_0v,P^{-\frac{1}{2}}_0v)=\int\limits_{\mathcal{D}}\mathcal{E}_0(v,v)\,dx=
  \mu\|D_2v_2+D_3v_3\|^2+\mu\sum^{3}_{j=1}\left(\|D_2v_j\|^2+\|D_3v_j\|^2\right)=
\]
\[
  =\mu\|D_2v_2+D_3v_3\|^2+\mu\left(\|D_2v_1\|^2+\|D_3v_1\|^2+\|D_2v_2\|^2+\|D_3v_2\|^2+
  \|D_2v_3\|^2+\|D_3v_3\|^2\right).
\]
\[
  (A_1\zeta,\zeta)=(\mathcal{A}v,v)=\mu\|v_1\|^2+\mu\sum^3_{j=1}\|v_j\|^2=\mu\|v_1\|^2+
  \mu\left(\|v_1\|^2+\|v_2\|^2+\|v_3\|^2\right).
\]
  From these relations we find that
\[
  4\tau^2(A_1\zeta,\zeta)(\zeta,\zeta)-(B_1\zeta,\zeta)^2=4\tau^2
  [\:\mu\|v_1\|^2+\mu\sum^3_{j=1}\|v_j\|^2\,]\,\times
\]
\[
  \times\,[\:\mu\|D_2v_2+D_3v_3\|^2+\mu\sum^3_{j=1}
  \left(\|D_2v_j\|^2+\|D_3v_j\|^2\right)\,]-4\mu^2[\,\mathop{\text{Re}}(iD_2v_2+iD_3v_3,v_1)\,]^2=
\]
\[
  =4\mu^2\tau^2[\,\|v_1\|^2\|D_2v_2+D_3v_3\|^2\!+\|D_2v_2+D_3v_3\|^2
  \!\sum^3_{j=1}\|v_j\|^2\!+\|v_1\|^2\!\sum^3_{j=1}\left(\|D_2v_j\|^2\!+\|D_3v_j\|^2\right)\!+
\]
\[
  +\sum^3_{j=1}\|v_j\|^2\sum^3_{j=1}\left(\|D_2v_j\|^2+\|D_3v_j\|^2\right)\,]-4\mu^2
  [\,\mathop{\text{Re}}(iD_2v_2+iD_3v_3,v_1)\,]^2=
\]
\[
  =4\mu^2[\,(\tau^2-1)\|v_1\|^2\|D_2v_2+D_3v_3\|^2+\tau^2\|D_2v_2+D_3v_3\|^2\sum^3_{j=1}\|v_j\|^2+
\]
\[
  +\,\tau^2\|v_1\|^2\sum^3_{j=1}\left(\|D_2v_j\|^2+\|D_3v_j\|^2\right)+\tau^2\sum^3_{j=1}\|v_j\|^2
  \sum^3_{j=1}\left(\|D_2v_j\|^2+\|D_3v_j\|^2\right)]+
\]

\[
  +\,4\mu^2[\,\|v_1\|^2\|D_2v_2+D_3v_3\|^2-\mathop{\text{Re}}(iD_2v_2+iD_3v_3,v_1)\,]^2\geq
\]

\[
  \geq
  4\mu^2\tau^2[\,\|D_2v_2+D_3v_3\|^2\sum^3_{j=2}\|v_j\|^2+\|v_1\|^2\sum^3_{j=1}\left(\|D_2v_j\|^2+
  \|D_3v_j\|^2\right)+
\]
\[
  +\sum^3_{j=1}\|v_j\|^2\sum^3_{j=1}\left(\|D_2v_j\|^2+\|D_3v_j\|^2\right)]+
  4[\,\mu^2(\tau^2-1)+\mu^2\tau^2\,]\|v_1\|^2\|D_2v_2+D_3v_3\|^2\geq 0,
\]
  if \, $\mu^2(\tau^2-1)+\mu^2\tau^2\geq 0$. From here it follows that inequality
  (\ref{n32}) holds for \, $\frac{1}{2}\leq\tau^2<1$.\\

  Now we considerer the quadratic pencil $L_0(\alpha)$. If $v=P^{\frac{1}{2}}\xi$, then, taking into
  account the previous results and the next relation:
\begin{equation} \label{n33}
  (iD_kv_1,v_k)=-(v_1,in_kv_k)_0+(v_1,iD_kv_k),\, k=2,3,
\end{equation}
  we find the next relations:
\[
  -(B_0\xi,\xi)=((B+Q)\xi,\xi)=(\mathcal{B}v,v)+(-i\mathcal{N}v,v)=
\]
\[
  =2\mu\mathop{\text{Re}}(iD_2v_2+iD_3v_3,v_1)+i\mu[\,(v_1,n_2v_2)_0+(v_1,n_3v_3)_0]
  -i\mu[\,(n_2v_2,v_1)_0+(n_3v_3,v_1)_0]=
\smallskip
\]
\[
  =2\mu\mathop{\text{Re}}(iD_2v_2+iD_3v_3,v_1)-2\mu\mathop{\text{Re}}(in_2v_2+in_3v_3,v_1)_0=
\smallskip
\]
\[
  =2\mu[\mathop{\text{Re}}(v_2,iD_2v_1)+\mathop{\text{Re}}(v_3,iD_3v_1)],
\smallskip
\]
  where the operators $B=P^{\frac{1}{2}}\mathcal{B}P^{\frac{1}{2}}$ \,and\,
  $Q\in\mathfrak{G}_l,\,l>3$, see \cite{KAGyOMB81}.

  Besides, we have the next relations:
\[
  ((I-P)\xi,\xi)=(P^{-\frac{1}{2}}v,P^{-\frac{1}{2}}v)-(v,v)=\int\limits_\mathcal{D}
  \mathcal{E}_0(v,v)\,dx=
\]
\[
  =\mu\|D_3v_2+D_2v_3\|^2+\mu\left(\|D_2v_1\|^2+\|D_3v_1\|^2\right)+2\mu\left(\|D_2v_2\|^2+\|D_3v_3\|^2\right),
\]
and
\[
  (A_0\xi,\xi)=\mu\|v_1\|^2+\mu\sum^3_{j=1}\|v_j\|^2=\mu\|v_1\|^2+\mu\left(\|v_1\|^2+\|v_2\|^2+\|v_3\|^2\right),
\]
  where the operator $A_0=P^{\frac{1}{2}}\mathcal{A}P^{\frac{1}{2}}$, see \cite{KAGyOMB81}.\\

  We denote \:
  $\text{I}(\xi)=4(A_0\xi,\xi)((I-P)\xi,\xi)-(B_0\xi,\xi)^2.$

  This expression is the discriminant of the equation $(L_0(\alpha)\xi,\xi)$,
  taken with the oppositive sign. We recall that
\[
  \frac{1}{4}\text{I}(\xi)=[\:\mu\|v_1\|^2+\mu\sum^3_{j=1}\|v_j\|^2\,][\:\mu\|D_3v_2+D_2v_3\|^2
  +\mu\left(\|D_2v_1\|^2+\|D_3v_1\|^2\right)+
\]
\[
  +\,2\mu\left(\|D_2v_2\|^2+\|D_3v_3\|^2\right)]-[\,\mu\mathop{\text{Re}}(v_2,iD_2v_1)+
  \mu\mathop{\text{Re}}(v_3,iD_3v_1)]^2=
\]
\smallskip
\[
  =\,\mu^2[\,\|v_1\|^2\|D_3v_2+D_2v_3\|^2+\|v_1\|^2\left(\|D_2v_1\|^2+\|D_3v_1\|^2\right)
  +2\|v_1\|^2\left(\|D_2v_2\|^2+\|D_3v_3\|^2\right)+
\]
\[
  +\,\|D_3v_2+D_2v_3\|^2\sum^3_{j=1}\|v_j\|^2+\sum^3_{j=1}\|v_j\|^2\left(\|D_2v_1\|^2+\|D_3v_1\|^2\right)+
\]
\[
  +\,2\sum^3_{j=1}\|v_j\|^2\left(\|D_2v_2\|^2+\|D_3v_3\|^2\right)]-\mu^2[\,\mathop{\text{Re}}(v_2,iD_2v_1)+
  \mathop{\text{Re}}(v_3,iD_3v_1)]^2\geq
\]
\[
  \geq\,\mu^2[\,\|v_1\|^2\|D_3v_2+D_2v_3\|^2+\|v_1\|^2\left(\|D_2v_1\|^2+\|D_3v_1\|^2\right)+
\]
\[
  +\,2\|v_1\|^2\left(\|D_2v_2\|^2+\|D_3v_3\|^2\right)
  +\,\|D_3v_2+D_2v_3\|^2\sum^3_{j=1}\|v_j\|^2+
\]
\[
  +\,\|v_1\|^2\left(\|D_2v_1\|^2+\|D_3v_1\|^2\right)+
  2\sum^3_{j=1}\|v_j\|^2\left(\|D_2v_2\|^2+\|D_3v_3\|^2\right)+
\]
\[
  +\sum^3_{j=2}\|v_j\|^2\left(\|D_2v_1\|^2+\|D_3v_1\|^2\right)-
  \left(\|v_2\|^2+\|v_3\|^2\right)\left(\|D_2v_1\|^2+\|D_3v_1\|^2\right)]\geq
\]
\[
 \geq\,\mu^2[\,\|v_1\|^2\|D_3v_2+D_2v_3\|^2+\,2\|v_1\|^2\left(\|D_2v_1\|^2\!+\|D_3v_1\|^2\right)+
 \,2\|v_1\|^2\left(\|D_2v_2\|^2\!+\|D_3v_3\|^2\right)+
\]
\begin{equation} \label{n34}
  +\,\sum^3_{j=1}\|v_j\|^2\|D_3v_2+D_2v_3\|^2
  +2\sum^3_{j=1}\|v_j\|^2\left(\|D_2v_2\|^2+\|D_3v_3\|^2\right)\geq 0.
\end{equation}

  Then the inequality (\ref{n34}) always  follows from the value $\mu>0$.
  Thus, $\text{I}(\xi)\geq 0$ for any function $\xi\in L_2(\mathcal{D})$.
  From the proof is clear that $\text{I}(\xi)=0$ only if
\[
  D_2v_1=D_3v_1=D_2v_2=D_3v_3=D_3v_2+D_2v_3=0.
\]
  This is possible only for
\[
  v_1=m_1; \quad v_2=m_2+m_4x_3; \quad v_3=m_3-m_4x_2,
\]
  where $m_i$ -- are arbitrary constants. By straightforward verification we can
  see that the vectors
\begin{align} \label{n35}
  v_0=(1,0,0); \: u_0=(0,1,0); \: w_0=(0,0,1); \; g_0=(0,x_3,-x_2),
\end{align}
  are indeed eigenvectors of the spectral problem $\mathcal{L}_0(\alpha)$,
  corresponding to the eigenvalue for $\alpha=0$ and from has been
  said above there follows that there are not other independent
  eigenvector for $\alpha=0$. The lemma is proved.\\

\medskip
\textbf{4.~Fundamental theorems for the pencils
  $\boldsymbol{L_\omega(\alpha)}$ and $\boldsymbol{L^0_\omega(\alpha)}$}\\

  We mention the following results. If $v\rightarrow \overline {v\mathstrut},\,v\in
  L_2(\mathcal{D})$~--- is the operator of passing to the conjugate function,
  then we have the equalities

\[
  \overline{\mathcal{L}_\omega(\alpha)v\mathstrut}=\mathcal{L}_\omega(-\overline
  {\alpha\mathstrut})\overline{v\mathstrut}, \: \, \text{where} \: \, v\in
  \mathscr{D}(\mathcal{L}_\omega(\alpha)) \; \; \text{and} \; \;
  \overline{\mathcal{L}^0_\omega(\alpha)v\mathstrut}=\mathcal{L}^0_\omega(-\overline
  {\alpha\mathstrut})\overline{v\mathstrut}, \: \: \text{where} \: \: v\in
  \mathscr{D}(\mathcal{L}^0_\omega(\alpha)).
\]

  If $\xi\rightarrow\overline{\xi\mathstrut}, \: \xi\in L_2(\mathcal{D})$~--- is the operator of passing
  to the conjugate function, then we have the equalities

\[
  \overline{L_\omega(\alpha)\xi\mathstrut}=L_\omega(-\overline{\alpha\mathstrut})
  \overline{\xi\mathstrut} \; \; \: \text{and} \; \; \:
  \overline{L^0_\omega(\alpha)\xi\mathstrut}=
  L^0_\omega(-\overline{\alpha\mathstrut})\overline{\xi\mathstrut}.
\medskip
\]

  From this equalities it is clear that if $\xi_n$ is an eigenvector
  of the quadratic pencil $L_\omega(\alpha)$ (or $L^0_\omega(\alpha)$) and
  it corresponds to the eigenvalue $\alpha_n$, then $\overline{\xi_n\mathstrut}$
  is also an eigenvector of the quadratic pencil $L_\omega(\alpha)$ (respectively $L^0_\omega(\alpha)$),
  corresponding to the eigenvalue $-\overline{\alpha\mathstrut}$.
  Taking into account that the quadratics pencils $L_\omega(\alpha)$ and $L^0_\omega(\alpha)$
  are self-adjoint, we obtain that their spectra $\sigma(L_\omega)$ and $\sigma(L^0_\omega)$
  are symmetrically situated relative to the axis and the origin.
  From the theorem~1.5 of the work \cite{AMCyVMI64} there follows that these spectra
  consist only of eigenvalues of finite multiplicity with a unique possible accumulation point
  at the infinity.\\

  The facts established above and lemma \ref{L2}, can bee summed up in the form of the following
  theorems:

\begin{theorem} \label{T1}
  For the isotropic structure, when the Lame constant are $\mu>0$ and  $\lambda=0$,
  the spectra $\sigma(L_\omega)$ and $\sigma(L^0_\omega)$ of the quadratic pencils $L_\omega(\alpha)$ and
  $L^0_\omega(\alpha)$, when  $\omega^2\geq 0$, consist of eigenvalues $\alpha_n$
  of finite multiplicity, symmetrical situated relative to the real axis and the origin, and, with the exception
  of a finite number of points, are in arbitrary in small angles,
  adjoining the imaginary axis. In addition, for $\omega=0$ the quadratic pencil
  $L^0_0(\alpha)$ does not have spectral points on the real axis, while the quadratic pencil $L_0(\alpha)$
  has only one spectral point $\alpha=0$, to which there correspond four linear independent
  vectors: $P^{-\frac{1}{2}}v_0, \; P^{-\frac{1}{2}}u_0, \; P^{-\frac{1}{2}}w_0, \; P^{-\frac{1}{2}}g_0$,
  where vectors $v_0, \; u_0, \; w_0, \; g_0$ are defined in (\ref{n35}) and are eigenvectors
  of the spectral problem $\mathcal{L}_0(\alpha)$, corresponding to the point $\alpha=0$.
  For $\omega^2\geq 0$ the quadratic pencils $L_\omega(\alpha)$ and $L^0_\omega(\alpha)$ admit in the domain
  $\Omega_{\varepsilon,N}$ the resolvent estimates
  \begin{align} \label{n36}
  \|P^{\frac{1}{2}}L^{-1}_\omega(\alpha)\|\leq c(\varepsilon,N)|\alpha|^{-1} \; \text{and} \;
  \; \|P^{\frac{1}{2}}_0L^0_\omega(\alpha)^{-1}\|\leq
  c(\varepsilon,N)|\alpha|^{-1}
  \end{align}
\end{theorem}

\begin{theorem} \label{T2}
  For the isotropic structure, when the Lame constants are $\mu>0$ and
  $\lambda=0$, the system of all eigen-and associated vectors of each of the
  spectral problems $\mathcal{L}_\omega(\alpha)$ and $\mathcal{L}^0_\omega(\alpha)$,
  repeatedly complete in the spaces $W^1_2(\mathcal{D})$ and
  $\overset{0}{W}{^1_2(\mathcal{D})}$, respectively. The one-halves
  of the systems of eigen-and associated vectors of each of the spectral
  problems $\mathcal{L}_\omega(\alpha)$ and $\mathcal{L}^0_\omega(\alpha)$
  are complete and minimal in the spaces $W^1_2(\mathcal{D})$ and $\overset{0}{W}{^1_2(\mathcal{D})}$
  respectively. In addition, the one-half of the eigen-and associated vectors
  of the spectral problems $\mathcal{L}_\omega(\alpha)$ and $\mathcal{L}^0_\omega(\alpha)$
  is minimal (and of course, complete) in the space
  $L_2(\mathcal{D})$.
\end{theorem}

\textbf{Proof.} The first assertions follows from theorem~2.3
  \cite{KAGyOMB81} and from the manner in which the eigen-and associated vectors
  of the spectral problems $\mathcal{L}_\omega(\alpha) \, (\mathcal{L}^0_\omega(\alpha))$
  and of the quadratic pencils $L_\omega(\alpha)$ (respectively $L^0_\omega(\alpha)$).
  are related themselves. The last assertion is obtained by using the remark 2.3 \cite{KAGyOMB81}.\\

\bigskip

\textbf{5.~Some applications to the spectral problems
  $\boldsymbol{\mathcal{L}_\omega(\alpha)}$ and $\boldsymbol{\mathcal{L}^0_\omega(\alpha)}$}\\

\textbf{5.1~Determination of the roots vectors for the spectral
  problem $\boldsymbol{\mathcal{L}_0(\omega)}$}

\bigskip

  In the static case, when $\omega=0$, according to
  theorem~\ref{T1}, the spectral problem $\mathcal{L}_0$, when the
  Lame constants are $\mu>0$ and  $\lambda=0$, has on the real axis only one eigenvalue
  $\alpha=0$, to which there correspond four independent
  eigenvectors. the Jordan chains corresponding to this eigenvalue $\alpha=0$,
  can be exhibited explicitly, as was done in the paper \cite{KAGyOMB81},
  when the Lame constants are $\mu>0$ and $\lambda>0$.

\begin{theorem} \label{T3}
  For the isotropic structure, when the
  Lame constants are $\mu>0$ and $\lambda=0$, the spectral problem
  $\mathcal{L}_0(\alpha)$, at the point $\alpha=0$, there correspond two
  chains of eigen-and associated vectors of length 2 and two chains
  of eigen-and associated vectors of length 4. The projections of
  these eigen-and associated vectors on to space $L_2(\mathcal{D})$
  have the form (the index 0 corresponds to the eigenvector, and the
  subsequent indices correspond to associated vectors):
  \begin{align} \label{n37}
    \begin{matrix}
    &  v_0=(1,0,0), &  v_1=(0,x_3,-x_2);    &\\
\\
    u_0=(0,1,0),    &  u_1=i(-x_2+k_2,0,0), &  u_2=(0,x_3,-x_2), &  u_3=(u_{31},0,0);\\
\\
    w_0=(0,0,1),  &  w_1=i(-x_3+k_3,0,0), &  w_2=(0,x_3,-x_2), &  w_3=(w_{31},0,0);\\
\\
    &  g_0=(0,x_3,-x_2),  &  g_1=(g_{11},0,0),  &
    \end{matrix}
  \end{align}
  where the constants $k_2$ and $k_3$ depend on the domain $\mathcal{D}$, and
  are defined by:
\[
  \int\limits_\mathcal{D}(x_2-k_2)\,dx=0 \; \; \text{and} \; \; \int\limits_\mathcal{D}(x_3-k_3)\,dx=0,
\]
  the functions $u_{31}, \, w_{31}\; \text{and} \; \; g_{11}$ are solutions of the following
  boundary problems:

  \begin{align*}
  &   \left \{
          \begin{matrix}
          \qquad\mu\left(D^2_2+D^2_3\right)u_{31} & = & -2i\mu(x_2-k_2) \; \: \text{in} \: \; \mathcal{D}, \\
          \left.{\mu(n_2D_2+n_3D_3)u_{31}}\right|_{\partial\mathcal{D}} & = & -i\mu(n_2x_3-n_3x_2),
          \end{matrix}
       \right.\\
\\
  &   \left \{
          \begin{matrix}
          \qquad\mu\left(D^2_2+D^2_3\right)w_{31} & = & -2i\mu(x_3-k_3) \; \: \text{in} \: \; \mathcal{D}, \\
          \left.{\mu(n_2D_2+n_3D_3)w_{31}}\right|_{\partial\mathcal{D}} & = & -i\mu(n_2x_3-n_3x_2), \notag
          \end{matrix}
       \right.\\
\\
&   \left\{
          \begin{matrix}
          \qquad\mu\left(D^2_2+D^2_3\right)g_{11} & = & 0 \qquad \text{in} \qquad \mathcal{D},  \\
          \left.{\mu(n_2D_2+n_3D_3)g_{11}}\right|_{\partial\mathcal{D}} & = & -i\mu(n_2x_3-n_3x_2),
          \end{matrix}
       \right.
  \end{align*}
\end{theorem}

\bigskip
\textbf{Proof.} In the lemma~\ref{L2} has bee proved that the
  spectral problem $\mathcal{L}_(\alpha)$ at the point $\alpha=0$ does have
  other eigenvectors in addition to $v_0, \, u_0, \, w_0 \; \text{and} \; g_0$.
  we carry out the computation for all the eigenvectors, when the
  Lame constants are $\mu>0$ and $\lambda=0$.
\bigskip
\\
  a) Let $v_1$~--- be an associated vector to be eigenvector $v_0$,
  then $v_1$ can be defined by the system of equations:
\[
  \mathcal{C}_0(0)v_1+\frac{\partial}{\partial\alpha}\mathcal{C}_0(0)v_0=
\]
\[
  =-\mu \! \!
\begin{array} \lgroup{ccc}\rgroup
   {\! \! \!}  D^2_2+ D^2_3  &        0       &         0                 \\
                    0        &  2D^2_2+D^2_3  &      D_2D_3                \\
                    0        &     D_2D_3     &   D^2_2+2D^2_3  {\! \! \!}
\end{array} \! \! \! \!
\begin{array}\lgroup{c}\rgroup
  {\! \! \!}  v_{11} {\! \! \!}  \\
  {\! \! \!}  v_{12} {\! \! \!}  \\
  {\! \! \!}  v_{13} {\! \! \!}
\end{array} \!
  \! -i\mu \! \!
\begin{array}\lgroup{ccc}\rgroup
   {\! \! \!} 0   &   D_2   &   D_3 {\! \! \!}  \\
   {\! \! \!} D_2 &    0    &    0  {\! \! \!}   \\
   {\! \! \!} D_3 &    0    &    0  {\! \! \!}
\end{array} \! \! \! \!
\begin{array}\lgroup{c}\rgroup
   {\! \! \!} 1 {\! \! \!} \\
   {\! \! \!} 0 {\! \! \!} \\
   {\! \! \!} 0 {\! \! \!}
\end{array} \! =
\]

\medskip
\[
  =-\mu \!
\begin{array}\lgroup{c}\rgroup
    \left(D^2_2+D^2_3\right)v_{11}  \\
   {\! \! \!} 2D^2_2v_{12}+D^2_3v_{12}+D_2D_3v_{13} {\! \! \!} \\
   {\! \! \!} D_2D_3v_{12}+D^2_2v_{13}+2D^2_3v_{13} {\! \! \!}
\end{array} {\! \! } = \!
\begin{array}\lgroup{c}\rgroup
  {\! \! \!} 0 {\! \! \!} \\
  {\! \! \!} 0 {\! \! \!} \\
  {\! \! \!} 0 {\! \! \!}
\end{array};
\]
\[
  \left.{\mathcal{U}(0)v_1+\frac{\partial}{\partial\alpha}\mathcal{U}(0)v_0}
  \right |_{\partial\mathcal{D}}=
\]
\[
  = \! \mu \!
\begin{array} \lgroup{ccc}\rgroup
  {\! \! \!} n_2D_2+n_3D_3    &           0        &       0            \\
         0          &   2n_2D_2+n_3D_3   &    n_3D_2            \\
         0          &        n_2D_3      & n_2D_2+2n_3D_3 {\! \! \!}
\end{array} \! \! \! \!
\begin{array}\lgroup{c}\rgroup
  {\! \! \!} v_{11} {\! \! \!} \\
  {\! \! \!} v_{12} {\! \! \!} \\
  {\! \! \!} v_{13} {\! \! \!}
\end{array}
   \! \! + i\mu \!
\begin{array} \lgroup{ccc}\rgroup
  {\! \! \!} 0 &    n_2    &    n_3 {\! \! \!} \\
  {\! \! \!} 0 &     0     &   0    {\! \! \!} \\
  {\! \! \!} 0 &     0     &   0    {\! \! \!}
\end{array} \! \! \! \!
\begin{array}\lgroup{c}\rgroup
  {\! \! \!} 1 {\! \! \!}  \\
  {\! \! \!} 0 {\! \! \!}  \\
  {\! \! \!} 0 {\! \! \!}
\end{array} {\! \!} =
\]

\medskip
\[
   {\! \! } = \mu {\! \! }
\begin{array}\lgroup{c}\rgroup
                          (n_2D_2+n_3D_3)v_{11}                  \\
   {\! \! \!} 2n_2D_2v_{12}+n_3D_3v_{12}+n_3D_2v_{13}  {\! \! \!} \\
   {\! \! \!} n_2D_3v_{12}+n_2D_2v_{13}+2n_3D_3v_{13} {\! \! \!}
\end{array}_{\!0} {\! \!}={\! \!}
\begin{array}\lgroup{c}\rgroup
  {\! \! \!} 0 {\! \! \!}  \\
  {\! \! \!} 0 {\! \! \!}  \\
  {\! \! \!} 0 {\! \! \!}
\end{array}.
\]\\
  From this system of equations, we obtain that the associated vector $v_1$
  has the form:
\[
  v_1=(0,x_3,-x_2).
\]

  Let $v_2$~--- be the next associated eigenvector, then $v_2$ can be defined
  by the system of equations:

\[
  \mathcal{C}_0(0)v_2+\frac{\partial}{\partial\alpha}\mathcal{C}_0(0)v_1+
  \frac{1}{2}\frac{\partial^2}{\partial\alpha^2}\mathcal{C}_0(0)v_0=
\]
\[
  =-\mu \! \!
\begin{array} \lgroup{ccc}\rgroup
   {\! \! \!}  D^2_2+ D^2_3  &        0       &         0                 \\
                    0        &  2D^2_2+D^2_3  &      D_2D_3                \\
                    0        &     D_2D_3     &   D^2_2+2D^2_3  {\! \! \!}
\end{array} \! \! \! \!
\begin{array}\lgroup{c}\rgroup
  {\! \! \!}  v_{21} {\! \! \!}  \\
  {\! \! \!}  v_{22} {\! \! \!}  \\
  {\! \! \!}  v_{23} {\! \! \!}
\end{array} \!
  \! -i\mu \! \!
\begin{array}\lgroup{ccc}\rgroup
   {\! \! \!} 0   &   D_2   &   D_3 {\! \! \!}  \\
   {\! \! \!} D_2 &    0    &    0  {\! \! \!}   \\
   {\! \! \!} D_3 &    0    &    0  {\! \! \!}
\end{array} \! \! \! \!
\begin{array}\lgroup{c}\rgroup
   {\! \! \!}            0   {\! \! \!} \\
   {\! \! \!}           ix_3 {\! \! \!} \\
   {\! \! \! \! \! \!} -ix_2 {\! \! \!}
\end{array} \! +
\]

\medskip
\[
  +\, \mu {\! \!}
\begin{array}\lgroup{ccc}\rgroup
   {\! \! \!} 2  &   0  &   0  {\! \! \!} \\
   {\! \! \!} 0  &   1  &   0  {\! \! \!} \\
   {\! \! \!} 0  &   0  &   1  {\! \! \!}
\end{array} {\! \! \! \! \!}
\begin{array}\lgroup{c}\rgroup
  {\! \! \!} 1 {\! \! \!}  \\
  {\! \! \!} 0 {\! \! \!}  \\
  {\! \! \!} 0 {\! \! \!}
\end{array} {\! \!}
  =-\mu {\! \!}
\begin{array}\lgroup{c}\rgroup
    \left(D^2_2+D^2_3\right)v_{21}-2  \\
   {\! \! \!} 2D^2_2v_{22}+D^2_3v_{22}+D_2D_3v_{23} {\! \! \!} \\
   {\! \! \!} D_2D_3v_{22}+D^2_2v_{23}+2D^2_3v_{23} {\! \! \!}
\end{array} {\! \! } = {\! \!}
\begin{array}\lgroup{c}\rgroup
  {\! \! \!} 0 {\! \! \!} \\
  {\! \! \!} 0 {\! \! \!} \\
  {\! \! \!} 0 {\! \! \!}
\end{array}\! ;
\]

\bigskip
\[
  \left.{\mathcal{U}(0)v_2+\frac{\partial}{\partial\alpha}\mathcal{U}(0)v_1}
  \right|_{\partial\mathcal{D}}=
\]
\[
  = \! \mu \!
\begin{array} \lgroup{ccc}\rgroup
  {\! \! \!} n_2D_2+n_3D_3    &           0        &       0            \\
         0          &   2n_2D_2+n_3D_3   &    n_3D_2            \\
         0          &        n_2D_3      & n_2D_2+2n_3D_3 {\! \! \!}
\end{array} \! \! \! \!
\begin{array}\lgroup{c}\rgroup
  {\! \! \!} v_{21} {\! \! \!} \\
  {\! \! \!} v_{22} {\! \! \!} \\
  {\! \! \!} v_{23} {\! \! \!}
\end{array}
   \! \! + i\mu \!
\begin{array} \lgroup{ccc}\rgroup
  {\! \! \!} 0 &    n_2    &    n_3 {\! \! \!} \\
  {\! \! \!} 0 &     0     &   0    {\! \! \!} \\
  {\! \! \!} 0 &     0     &   0    {\! \! \!}
\end{array} \! \! \! \!
\begin{array}\lgroup{c}\rgroup
   {\! \! \!}            0   {\! \! \!} \\
   {\! \! \!}           ix_3 {\! \! \!} \\
   {\! \! \! \! \! \!} -ix_2 {\! \! \!}
\end{array}{\! \!} =
\]

\medskip
\[
   {\! \! } = \mu {\! \! }
\begin{array}\lgroup{c}\rgroup
      {\! \! \!}  (n_2D_2+n_3D_3)v_{21}-i(n_2x_3-n_3x_2) {\! \! \!}   \\
                  2n_2D_2v_{22}+n_3D_3v_{22}+n_3D_2v_{23}             \\
                  n_2D_3v_{22}+n_2D_2v_{23}+2n_3D_3v_{23}
\end{array}_{\!0} {\! \!}={\! \!}
\begin{array}\lgroup{c}\rgroup
  {\! \! \!} 0 {\! \! \!}  \\
  {\! \! \!} 0 {\! \! \!}  \\
  {\! \! \!} 0 {\! \! \!}
\end{array}.
\]\\
  From this system of equations, we examine the next boundary problem:
\begin{align} \label{n38}
      \left \{
        \begin{matrix}
          \qquad\mu\left(D^2_2+D^2_3\right)v_{21} & = &   2\mu
          \\
          \left.{\mu(n_2D_2+n_3D_3)v_{21}}\right|_{\partial\mathcal{D}} & = & i\mu(n_2x_3-n_3x_2).
        \end{matrix}
       \right.
\end{align}
  Thus, in order to find $v_{21}$, one has to prove the solvability of the problem (\ref{n38}).
  This problem is solvable only, when we have the next equality:
\[
  -2\mu\!\int\limits_{\mathcal{D}}\! dx+(i\mu(n_2x_3-n_3x_2),1)_0=0.
\]
  From the relation (\ref{n33}) we obtain that $(i\mu(n_2x_3-n_3x_2),1)_0=0$, then
\[
  -2\mu\!\int\limits_{\mathcal{D}}\! dx+(i\mu(n_2x_3-n_3x_2),1)_0=-2\mu\!\int
  \limits_{\mathcal{D}}\! dx<0.
\]
  For this inequality, the boundary problem (\ref{n38})
  and at the same time the system of equations for the vector are not
  solvable. Consequently, the chain of the vectors:
\[
  v_0=(1,0,0), \; \; \emph{and} \, \; \; v_1=(0,x_3,-x_2)\text{~---} \,\emph{is maximal}.
\]\\
\medskip
   b) Let $u_1$~--- be an associated vector to be eigenvector $u_ 0$, then $u_1$
   can be defined by the system of equations:
\[
  \mathcal{C}_0(0)u_1+\frac{\partial}{\partial\alpha}\mathcal{C}_0(0)u_0=
\]
\[
  =-\mu \! \!
\begin{array} \lgroup{ccc}\rgroup
   {\! \! \!}  D^2_2+ D^2_3  &        0       &         0                 \\
                    0        &  2D^2_2+D^2_3  &      D_2D_3                \\
                    0        &     D_2D_3     &   D^2_2+2D^2_3  {\! \! \!}
\end{array} \! \! \! \!
\begin{array}\lgroup{c}\rgroup
  {\! \! \!}  u_{11} {\! \! \!}  \\
  {\! \! \!}  u_{12} {\! \! \!}  \\
  {\! \! \!}  u_{13} {\! \! \!}
\end{array} \!
  \! -i\mu \! \!
\begin{array}\lgroup{ccc}\rgroup
   {\! \! \!} 0   &   D_2   &   D_3 {\! \! \!}  \\
   {\! \! \!} D_2 &    0    &    0  {\! \! \!}   \\
   {\! \! \!} D_3 &    0    &    0  {\! \! \!}
\end{array} \! \! \! \!
\begin{array}\lgroup{c}\rgroup
   {\! \! \!} 0 {\! \! \!} \\
   {\! \! \!} 1 {\! \! \!} \\
   {\! \! \!} 0 {\! \! \!}
\end{array} \! =
\]

\medskip
\[
  =-\mu \!
\begin{array}\lgroup{c}\rgroup
    \left(D^2_2+D^2_3\right)u_{11}  \\
   {\! \! \!} 2D^2_2u_{12}+D^2_3u_{12}+D_2D_3u_{13} {\! \! \!} \\
   {\! \! \!} D_2D_3u_{12}+D^2_2u_{13}+2D^2_3u_{13} {\! \! \!}
\end{array} {\! \! } = \!
\begin{array}\lgroup{c}\rgroup
  {\! \! \!} 0 {\! \! \!} \\
  {\! \! \!} 0 {\! \! \!} \\
  {\! \! \!} 0 {\! \! \!}
\end{array};
\]

\bigskip
\[
  \left.{\mathcal{U}(0)u_1+\frac{\partial}{\partial\alpha}\mathcal{U}(0)u_0}
  \right |_{\partial\mathcal{D}}=
\]
\[
  = \! \mu \!
\begin{array} \lgroup{ccc}\rgroup
  {\! \! \!} n_2D_2+n_3D_3    &           0        &       0            \\
         0          &   2n_2D_2+n_3D_3   &    n_3D_2            \\
         0          &        n_2D_3      & n_2D_2+2n_3D_3 {\! \! \!}
\end{array} \! \! \! \!
\begin{array}\lgroup{c}\rgroup
  {\! \! \!} u_{11} {\! \! \!} \\
  {\! \! \!} u_{12} {\! \! \!} \\
  {\! \! \!} u_{13} {\! \! \!}
\end{array}
   \! \! + i\mu \!
\begin{array} \lgroup{ccc}\rgroup
  {\! \! \!} 0 &    n_2    &    n_3 {\! \! \!} \\
  {\! \! \!} 0 &     0     &   0    {\! \! \!} \\
  {\! \! \!} 0 &     0     &   0    {\! \! \!}
\end{array} \! \! \! \!
\begin{array}\lgroup{c}\rgroup
  {\! \! \!} 0 {\! \! \!}  \\
  {\! \! \!} 1 {\! \! \!}  \\
  {\! \! \!} 0 {\! \! \!}
\end{array} {\! \!} =
\]

\medskip
\[
   {\! \! } = \mu {\! \! }
\begin{array}\lgroup{c}\rgroup
                     (n_2D_2+n_3D_3)u_{11}+in_2                  \\
   {\! \! \!} 2n_2D_2u_{12}+n_3D_3u_{12}+n_3D_2u_{13}  {\! \! \!} \\
   {\! \! \!} n_2D_3u_{12}+n_2D_2u_{13}+2n_3D_3u_{13} {\! \! \!}
\end{array}_{\!0} {\! \!}={\! \!}
\begin{array}\lgroup{c}\rgroup
  {\! \! \!} 0 {\! \! \!}  \\
  {\! \! \!} 0 {\! \! \!}  \\
  {\! \! \!} 0 {\! \! \!}
\end{array}.
\bigskip
\]
  From this system of equations, we obtain that the associated vector $u_1$ has the form:
\[
  u_1=i(-x_2+k_2,0,0), \, \emph{where}\; k_2 \text{~---} \, \emph{constant depends on the domain}
  \: \mathcal{D}.
\]

  Let $u_2$~--- be the next associated eigenvector, then $u_2$ can be defined
  by the system of equations:
\[
  \mathcal{C}_0(0)u_2+\frac{\partial}{\partial\alpha}\mathcal{C}_0(0)u_1+
  \frac{1}{2}\frac{\partial^2}{\partial\alpha^2}\mathcal{C}_0(0)u_0=
\]
\[
  =-\mu \! \!
\begin{array} \lgroup{ccc}\rgroup
   {\! \! \!}  D^2_2+ D^2_3  &        0       &         0                 \\
                    0        &  2D^2_2+D^2_3  &      D_2D_3                \\
                    0        &     D_2D_3     &   D^2_2+2D^2_3  {\! \! \!}
\end{array} \! \! \! \!
\begin{array}\lgroup{c}\rgroup
  {\! \! \!}  u_{21} {\! \! \!}  \\
  {\! \! \!}  u_{22} {\! \! \!}  \\
  {\! \! \!}  u_{23} {\! \! \!}
\end{array} \!
  \! -i\mu \! \!
\begin{array}\lgroup{ccc}\rgroup
   {\! \! \!} 0   &   D_2   &   D_3 {\! \! \!}  \\
   {\! \! \!} D_2 &    0    &    0  {\! \! \!}   \\
   {\! \! \!} D_3 &    0    &    0  {\! \! \!}
\end{array} {\! \! \! \!}
\begin{array}\lgroup{c}\rgroup
   {\! \! \! \! \!}  -\!i(x_2-k_2)   {\! \! \!} \\
                             0                  \\
                             0
\end{array} \! +
\]

\medskip
\[
  +\, \mu {\! \!}
\begin{array}\lgroup{ccc}\rgroup
   {\! \! \!} 2  &   0  &   0  {\! \! \!} \\
   {\! \! \!} 0  &   1  &   0  {\! \! \!} \\
   {\! \! \!} 0  &   0  &   1  {\! \! \!}
\end{array} {\! \! \! \! \!}
\begin{array}\lgroup{c}\rgroup
  {\! \! \!} 0 {\! \! \!}  \\
  {\! \! \!} 1 {\! \! \!}  \\
  {\! \! \!} 0 {\! \! \!}
\end{array} {\! \!}
  =-\mu {\! \!}
\begin{array}\lgroup{c}\rgroup
                   \left(D^2_2+D^2_3\right)u_{21}    \\
   {\! \! \!} 2D^2_2u_{22}+D^2_3u_{22}+D_2D_3u_{23} {\! \! \!} \\
   {\! \! \!} D_2D_3u_{22}+D^2_2u_{23}+2D^2_3u_{23} {\! \! \!}
\end{array} {\! \! } = {\! \!}
\begin{array}\lgroup{c}\rgroup
  {\! \! \!} 0 {\! \! \!} \\
  {\! \! \!} 0 {\! \! \!} \\
  {\! \! \!} 0 {\! \! \!}
\end{array}\! ;
\]
\[
  \left.{\mathcal{U}(0)u_2+\frac{\partial}{\partial\alpha}\mathcal{U}(0)u_1}
  \right|_{\partial\mathcal{D}}=
\]
\[
  = \! \mu \!
\begin{array} \lgroup{ccc}\rgroup
  {\! \! \!} n_2D_2+n_3D_3    &           0        &       0            \\
         0          &   2n_2D_2+n_3D_3   &    n_3D_2            \\
         0          &        n_2D_3      & n_2D_2+2n_3D_3 {\! \! \!}
\end{array} \! \! \! \!
\begin{array}\lgroup{c}\rgroup
  {\! \! \!} u_{21} {\! \! \!} \\
  {\! \! \!} u_{22} {\! \! \!} \\
  {\! \! \!} u_{23} {\! \! \!}
\end{array}+
\]
\medskip
\[
    + \, i\mu \!
\begin{array} \lgroup{ccc}\rgroup
  {\! \! \!} 0 &    n_2    &    n_3 {\! \! \!} \\
  {\! \! \!} 0 &     0     &   0    {\! \! \!} \\
  {\! \! \!} 0 &     0     &   0    {\! \! \!}
\end{array} \! \! \! \!
\begin{array}\lgroup{c}\rgroup
   {\! \! \! \! \!}  -\!i(x_2-k_2)   {\! \! \!} \\
                             0                  \\
                             0
\end{array} \!
   {\! \! } = \mu {\! \! }
\begin{array}\lgroup{c}\rgroup
                         (n_2D_2+n_3D_3)u_{21}   \\
  {\! \! \!}   2n_2D_2u_{22}+n_3D_3u_{22}+n_3D_2u_{23}  {\! \! \!}  \\
  {\! \! \!}   n_2D_3u_{22}+n_2D_2u_{23}+2n_3D_3u_{23} {\! \! \!}
\end{array}_{\!0} {\! \!}={\! \!}
\begin{array}\lgroup{c}\rgroup
  {\! \! \!} 0 {\! \! \!}  \\
  {\! \! \!} 0 {\! \! \!}  \\
  {\! \! \!} 0 {\! \! \!}
\end{array}.
\]\\
  From this system of equations, we obtain that the associated vector $u_2$
  has the form:
\[
  u_2=(0,x_3,-x_2).
\]
  Let $u_3$~--- be the next associated eigenvector, then $u_3$ can be defined
  by the system of equations:

\[
  \mathcal{C}_0(0)u_3+\frac{\partial}{\partial\alpha}\mathcal{C}_0(0)u_2+
  \frac{1}{2}\frac{\partial^2}{\partial\alpha^2}\mathcal{C}_0(0)u_1=
\]
\[
  =-\mu \! \!
\begin{array} \lgroup{ccc}\rgroup
   {\! \! \!}  D^2_2+ D^2_3  &        0       &         0                 \\
                    0        &  2D^2_2+D^2_3  &      D_2D_3                \\
                    0        &     D_2D_3     &   D^2_2+2D^2_3  {\! \! \!}
\end{array} \! \! \! \!
\begin{array}\lgroup{c}\rgroup
  {\! \! \!}  u_{31} {\! \! \!}  \\
  {\! \! \!}  u_{32} {\! \! \!}  \\
  {\! \! \!}  u_{33} {\! \! \!}
\end{array} \!
  \! -i\mu \! \!
\begin{array}\lgroup{ccc}\rgroup
   {\! \! \!} 0   &   D_2   &   D_3 {\! \! \!}  \\
   {\! \! \!} D_2 &    0    &    0  {\! \! \!}   \\
   {\! \! \!} D_3 &    0    &    0  {\! \! \!}
\end{array} {\! \! \! \!}
\begin{array}\lgroup{c}\rgroup
  {\! \! \!}         0  {\! \! \!}  \\
  {\! \! \!}        x_3 {\! \! \!}  \\
  {\! \! \! \! \!} -x_2 {\! \! \!}
\end{array} {\! \!}+
\]
\smallskip
\[
  +\, \mu {\! \!}
\begin{array}\lgroup{ccc}\rgroup
   {\! \! \!} 2  &   0  &   0  {\! \! \!} \\
   {\! \! \!} 0  &   1  &   0  {\! \! \!} \\
   {\! \! \!} 0  &   0  &   1  {\! \! \!}
\end{array} {\! \! \! \! \!}
\begin{array}\lgroup{c}\rgroup
  {\! \! \! \! \!} -i(x_2-k_2) {\! \! \!}  \\
  {\! \! \!}         0     {\! \! \!}  \\
  {\! \! \!}         0     {\! \! \!}
\end{array} {\! \!}
  =-\mu {\! \!}
\begin{array}\lgroup{c}\rgroup
   {\! \! \! \!} \left(D^2_2+D^2_3\right)u_{31}+2i(x_2-k_2) {\! \! \! \!}   \\
   {\! \! \!} 2D^2_2u_{32}+D^2_3u_{32}+D_2D_3u_{33} {\! \! \!} \\
   {\! \! \!} D_2D_3u_{32}+D^2_2u_{33}+2D^2_3u_{33} {\! \! \!}
\end{array} {\! \! } = {\! \!}
\begin{array}\lgroup{c}\rgroup
  {\! \! \!} 0 {\! \! \!} \\
  {\! \! \!} 0 {\! \! \!} \\
  {\! \! \!} 0 {\! \! \!}
\end{array}\! ;
\]

\bigskip
\[
  \left.{\mathcal{U}(0)u_3+\frac{\partial}{\partial\alpha}\mathcal{U}(0)u_2}
  \right|_{\partial\mathcal{D}}=
\]
\[
  = \! \mu \!
\begin{array} \lgroup{ccc}\rgroup
  {\! \! \!} n_2D_2+n_3D_3    &           0        &       0            \\
         0          &   2n_2D_2+n_3D_3   &    n_3D_2            \\
         0          &        n_2D_3      & n_2D_2+2n_3D_3 {\! \! \!}
\end{array} \! \! \! \!
\begin{array}\lgroup{c}\rgroup
  {\! \! \!} u_{31} {\! \! \!} \\
  {\! \! \!} u_{32} {\! \! \!} \\
  {\! \! \!} u_{33} {\! \! \!}
\end{array}
\]
\medskip
\[
   \! \! + \, i\mu \!
\begin{array} \lgroup{ccc}\rgroup
  {\! \! \!} 0 &    n_2    &    n_3 {\! \! \!} \\
  {\! \! \!} 0 &     0     &   0    {\! \! \!} \\
  {\! \! \!} 0 &     0     &   0    {\! \! \!}
\end{array} \! \! \! \!
\begin{array}\lgroup{c}\rgroup
   {\! \! \! \! }  0   {\! \! \!} \\
                   x_3             \\
    {\! \! \! \! \! \! } -x_2   {\! \! \! \! }
\end{array} \! \!
   {\! \! } = \mu {\! \! }
\begin{array}\lgroup{c}\rgroup
  {\! \! \!}   (n_2D_2+n_3D_3)u_{31}+i(n_2x_3-n_3x_2) {\! \! \!}  \\
  {\! \! \!}   2n_2D_2u_{32}+n_3D_3u_{32}+n_3D_2u_{33}  {\! \! \!}  \\
  {\! \! \!}   n_2D_3u_{32}+n_2D_2u_{33}+2n_3D_3u_{33} {\! \! \!}
\end{array}_{\!0} {\! \!}={\! \!}
\begin{array}\lgroup{c}\rgroup
  {\! \! \!} 0 {\! \! \!}  \\
  {\! \! \!} 0 {\! \! \!}  \\
  {\! \! \!} 0 {\! \! \!}
\end{array}.
\]\\
  We will look for the solution of the system of equations in the form $u_3=(u_{31},0,0)$.
  From this system of equations, we examine the next boundary problem:
\begin{align} \label{n39}
      \left \{
        \begin{matrix}
          \qquad\mu\left(D^2_2+D^2_3\right)u_{31} & = & -2i\mu(x_2-k_2)
          \\
          \left.{\mu(n_2D_2+n_3D_3)u_{31}}\right|_{\partial\mathcal{D}} & = & -i\mu(n_2x_3-n_3x_2).
        \end{matrix}
       \right.
\end{align}
  Thus, in order to find $u_{31}$, one has to prove the solvability of the problem (\ref{n39}).
  This problem is solvable only, when we have the next equality:
\[
  2\,i \mu\! \!\int\limits_\mathcal{D}(x_2-k_2)\,dx+(-i\mu(n_2x_3-n_3x_2),1)_0=0\,.
\]
  We have that $(i\mu(n_2x_3-n_3x_2),1)_0=0$, see (\ref{n33}), then
\[
  2\,i \mu\! \!\int\limits_\mathcal{D}(x_2-k_2)\,dx+(-i\mu(n_2x_3-n_3x_2),1)_0=
  2\,i \mu\! \!\int\limits_\mathcal{D}(x_2-k_2)\,dx.
\]
  Determining now the constant $k_2$ from the condition $\int\limits_\mathcal{D}(x_2-k_2)\,dx=0$,
  we find that the boundary problem (\ref{n39}) is solvable and thus, one can take vector
  $u_3=(u_{31},0,0)$, where $u_{31}$ is some solution of the problem (\ref{n39}).\\

  Let $u_4$~--- be the next associated eigenvector, then $u_4$ can be defined
  by the system of equations:

\[
  \mathcal{C}_0(0)u_4+\frac{\partial}{\partial\alpha}\mathcal{C}_0(0)u_3+
  \frac{1}{2}\frac{\partial^2}{\partial\alpha^2}\mathcal{C}_0(0)u_2=
\]
\[
  =-\mu \! \!
\begin{array} \lgroup{ccc}\rgroup
   {\! \! \!}  D^2_2+ D^2_3  &        0       &         0                 \\
                    0        &  2D^2_2+D^2_3  &      D_2D_3                \\
                    0        &     D_2D_3     &   D^2_2+2D^2_3  {\! \! \!}
\end{array} \! \! \! \!
\begin{array}\lgroup{c}\rgroup
  {\! \! \!}  u_{41} {\! \! \!}  \\
  {\! \! \!}  u_{42} {\! \! \!}  \\
  {\! \! \!}  u_{43} {\! \! \!}
\end{array} \!
  \! -i\mu \! \!
\begin{array}\lgroup{ccc}\rgroup
   {\! \! \!} 0   &   D_2   &   D_3 {\! \! \!}  \\
   {\! \! \!} D_2 &    0    &    0  {\! \! \!}   \\
   {\! \! \!} D_3 &    0    &    0  {\! \! \!}
\end{array} {\! \! \! \!}
\begin{array}\lgroup{c}\rgroup
  {\! \! \!}    u_{31}  {\! \! \!}  \\
  {\! \! \!}      0    {\! \! \!}  \\
  {\! \! \!}      0   {\! \! \!}
\end{array} {\! \!}+
\]
\smallskip
\[
  +\, \mu {\! \!}
\begin{array}\lgroup{ccc}\rgroup
   {\! \! \!} 2  &   0  &   0  {\! \! \!} \\
   {\! \! \!} 0  &   1  &   0  {\! \! \!} \\
   {\! \! \!} 0  &   0  &   1  {\! \! \!}
\end{array} {\! \! \! \! \!}
\begin{array}\lgroup{c}\rgroup
   {\! \! \! \! }  0   {\! \! \!} \\
                   x_3             \\
    {\! \! \! \! \! \! } -x_2   {\! \! \! \! }
\end{array} \! \!
 =-\mu {\! \!}
\begin{array}\lgroup{c}\rgroup
   {\! \! \! \!} \left(D^2_2+D^2_3\right)u_{41} {\! \! \! \!}   \\
   {\! \! \!} 2D^2_2u_{42}+D^2_3u_{42}+D_2D_3u_{43}+iD_2u_{31}-x_3 {\! \! \!} \\
   {\! \! \!} D_2D_3u_{42}+D^2_2u_{43}+2D^2_3u_{43}+iD_3u_{31}+x_2 {\! \! \!}
\end{array} {\! \! } = {\! \!}
\begin{array}\lgroup{c}\rgroup
  {\! \! \!} 0 {\! \! \!} \\
  {\! \! \!} 0 {\! \! \!} \\
  {\! \! \!} 0 {\! \! \!}
\end{array}\! ;
\]

\bigskip
\[
  \left.{\mathcal{U}(0)u_4+\frac{\partial}{\partial\alpha}\mathcal{U}(0)u_3}
  \right|_{\partial\mathcal{D}}=
\]
\[
  = \! \mu \!
\begin{array} \lgroup{ccc}\rgroup
  {\! \! \!} n_2D_2+n_3D_3    &           0        &       0            \\
         0          &   2n_2D_2+n_3D_3   &    n_3D_2            \\
         0          &        n_2D_3      & n_2D_2+2n_3D_3 {\! \! \!}
\end{array} \! \! \! \!
\begin{array}\lgroup{c}\rgroup
  {\! \! \!} u_{41} {\! \! \!} \\
  {\! \! \!} u_{42} {\! \! \!} \\
  {\! \! \!} u_{43} {\! \! \!}
\end{array}
   \! \! + \, i\mu \!
\begin{array} \lgroup{ccc}\rgroup
  {\! \! \!} 0 &    n_2    &    n_3 {\! \! \!} \\
  {\! \! \!} 0 &     0     &   0    {\! \! \!} \\
  {\! \! \!} 0 &     0     &   0    {\! \! \!}
\end{array} \! \! \! \!
\begin{array}\lgroup{c}\rgroup
  {\! \! \!}    u_{31}  {\! \! \!}  \\
  {\! \! \!}      0    {\! \! \!}  \\
  {\! \! \!}      0   {\! \! \!}
\end{array} {\! \!}=
\]
\medskip
\[
  = \mu
\begin{array}\lgroup{c}\rgroup
  {\! \! \!}   (n_2D_2+n_3D_3)u_{41} {\! \! \!}  \\
  {\! \! \!}   2n_2D_2u_{42}+n_3D_3u_{42}+n_3D_2u_{43}  {\! \! \!}  \\
  {\! \! \!}   n_2D_3u_{42}+n_2D_2u_{43}+2n_3D_3u_{43} {\! \! \!}
\end{array}_{\!0} {\! \!}={\! \!}
\begin{array}\lgroup{c}\rgroup
  {\! \! \!} 0 {\! \! \!}  \\
  {\! \! \!} 0 {\! \! \!}  \\
  {\! \! \!} 0 {\! \! \!}
\end{array}.
\]\\
  The system of equation have solution $u_4$ if the next boundary problem
  is solvable:
\begin{align} \label{n40}
\begin{matrix}
 & - \mu {\! \!}
  \begin{array}\lgroup{cc}\rgroup
     {\! \! \!} 2D^2_2+D^2_3   &     D_2D_3     \\
         D_2D_3             &  D^2_2+2D^2_3 {\! \! \!}
  \end{array} {\! \! \! \! \!}
  \begin{array}\lgroup{c}\rgroup
  {\! \! \!} u_{42} {\! \! \!} \\
  {\! \! \!} u_{43} {\! \! \!}
  \end{array}  {\! \! \!} =
  i\mu {\! \!}
  \begin{array}\lgroup{cc}\rgroup
  {\! \! \!} D_2  &  0  {\! \! \!}  \\
  {\! \! \!} D_3  &  0  {\! \! \!}
  \end{array} {\! \! \! \! \!}
  \begin{array}\lgroup{c}\rgroup
  {\! \! \!} u_{31} {\! \! \!} \\
  {\! \! \!}   0    {\! \! \!}
  \end{array}  {\! \! \!}
  -\mu {\! \!}
  \begin{array}\lgroup{cc}\rgroup
  {\! \! \!} 1  &  0  {\! \! \!}  \\
  {\! \! \!} 0  &  1  {\! \! \!}
  \end{array} {\! \! \! \! \!}
  \begin{array}\lgroup{c}\rgroup
  {\! \! \!}   x_3 {\! \! \!} \\
  {\! \! \! \! \!}  -x_2    {\! \! \!}
  \end{array}=f
\\
\smallskip
\\
 & \mu {\! \!}
  \begin{array}\lgroup{cc}\rgroup
     {\! \! \!} 2n_2D_2+n_3D_3   &     n_3D_2     \\
              n_2D_3             &  n_2D_2+2n_3D_3 {\! \! \!}
  \end{array} {\! \! \! \! \!}
  \begin{array}\lgroup{c}\rgroup
  {\! \! \!} u_{42} {\! \! \!} \\
  {\! \! \!} u_{43}    {\! \! \!}
  \end{array}_{\partial\mathcal{D}}  {\! \! \!} =
  -\,i  {\! \!}
  \begin{array}\lgroup{cc}\rgroup
  {\! \! \!} 0  &  0  {\! \! \!}  \\
  {\! \! \!} 0  &  0  {\! \! \!}
  \end{array} {\! \! \! \! \!}
  \begin{array}\lgroup{c}\rgroup
  {\! \! \!} u_{31} {\! \! \!} \\
  {\! \! \!}   0    {\! \! \!}
  \end{array} = h
\end{matrix}
\end{align}\\
  Therefore, the boundary problem (\ref{n40}) is solvable if and
  only  if we have the relations:

\[
  (f,e_k)+(h,e_k)_0=0, \, k=1,2,3,
\smallskip
\]
  where $e_1=(1,0)$, $e_2=(0,1)$ and $e_3=(x_3,-x_2)$ are the solutions of the homogeneous
  boundary problem (\ref{n40}) i.e., when $f=h=0$. We examine if this case is true for the vector
  $e_1=(1,0)$, then
\[
  (f,e_1)+(h,e_1)_0=i\mu(D_2u_{31},1)-\mu(x_3,1)=i\mu(D_2u_{31},D_2x_2)-\mu(x_3,D_2x_2)=
\]
\[
  =-i\mu(D^2_2u_{31},x_2)+\mu(in_2D_2u_{31},x_2)_0+\mu(D_2x_3,x_2)-\mu(n_2x_3,x_2)_0=
\smallskip
\]
\[
  =i\mu(D^2_3u_{31})-2\mu(x_2-k_2,x_2)+\mu(in_2D_2u_{31}-n_2x_3,x_2)_0=i\mu(D^2_3u_{31},x_2)-
\smallskip
\]
\[
  -2\mu(x_2-k_2,x_2)-\mu(in_3D_3u_{31}+n_3x_2,x_2)_0=
\smallskip
\]
\[
  =i\mu(D^2_3u_{31},x_2)-\mu(in_3D_3u_{31}+n_3x_2,x_2)-2\mu(x_2-k_2,x_2)=
\smallskip
\]
\[
  =i\mu(D^2_3u_{31},x_2)-i\mu(n_3D_3u_{31},x_2)_0-\mu(n_3x_2,x_2)_0-2\mu(x_2-k_2,x_2)=
\smallskip
\]
\[
  =i\mu(D^2_3u_{31},x_2)-i\mu(D^2_3u_{31},x_2)+\mu(D_3u_{31},iD_3x_2)-\mu(D_3x_2,x_2)-\mu(x_2,D_3x_2)-
\smallskip
\]
\[
  -2\mu(x_2-k_2,x_2)=-2\mu(x_2-k_2,x_2)=-2\mu\|x_2-k_2\|^2<0.
\smallskip
\]
  For this inequality the boundary problem (\ref{n40}) and at the
  same time the system of equations for the vector are not solvable.
  Consequently, the chain of the vectors:
\[
  u_0=(0,1,0), u_1=-i(x_2-k_2,0,0), u_2=(0,x_3,-x_2), \, \emph{and} \,
  \, u_3=(u_{31},0,0)\text{~---}\,\emph{is maximal}.
\]

  c) Let $w_1$~--- be an associated vector to be eigenvector $w_ 0$,
  then $w_1$ can be defined by the system of equations:
\[
  \mathcal{C}_0(0)w_1+\frac{\partial}{\partial\alpha}\mathcal{C}_0(0)w_0=
\]
\[
  =-\mu \! \!
\begin{array} \lgroup{ccc}\rgroup
   {\! \! \!}  D^2_2+ D^2_3  &        0       &         0                 \\
                    0        &  2D^2_2+D^2_3  &      D_2D_3                \\
                    0        &     D_2D_3     &   D^2_2+2D^2_3  {\! \! \!}
\end{array} \! \! \! \!
\begin{array}\lgroup{c}\rgroup
  {\! \! \!}  w_{11} {\! \! \!}  \\
  {\! \! \!}  w_{12} {\! \! \!}  \\
  {\! \! \!}  w_{13} {\! \! \!}
\end{array} \!
  \! -i\mu \! \!
\begin{array}\lgroup{ccc}\rgroup
   {\! \! \!} 0   &   D_2   &   D_3 {\! \! \!}  \\
   {\! \! \!} D_2 &    0    &    0  {\! \! \!}   \\
   {\! \! \!} D_3 &    0    &    0  {\! \! \!}
\end{array} \! \! \! \!
\begin{array}\lgroup{c}\rgroup
   {\! \! \!} 0 {\! \! \!} \\
   {\! \! \!} 0 {\! \! \!} \\
   {\! \! \!} 1 {\! \! \!}
\end{array} \! =
\]

\medskip
\[
  =-\mu \!
\begin{array}\lgroup{c}\rgroup
    \left(D^2_2+D^2_3\right)w_{11}  \\
   {\! \! \!} 2D^2_2w_{12}+D^2_3w_{12}+D_2D_3w_{13} {\! \! \!} \\
   {\! \! \!} D_2D_3w_{12}+D^2_2w_{13}+2D^2_3w_{13} {\! \! \!}
\end{array} {\! \! } = \!
\begin{array}\lgroup{c}\rgroup
  {\! \! \!} 0 {\! \! \!} \\
  {\! \! \!} 0 {\! \! \!} \\
  {\! \! \!} 0 {\! \! \!}
\end{array};
\]

\bigskip
\[
  \left.{\mathcal{U}(0)w_1+\frac{\partial}{\partial\alpha}\mathcal{U}(0)w_0}
  \right |_{\partial\mathcal{D}}=
\]
\[
  = \! \mu \!
\begin{array} \lgroup{ccc}\rgroup
  {\! \! \!} n_2D_2+n_3D_3    &           0        &       0            \\
         0          &   2n_2D_2+n_3D_3   &    n_3D_2            \\
         0          &        n_2D_3      & n_2D_2+2n_3D_3 {\! \! \!}
\end{array} \! \! \! \!
\begin{array}\lgroup{c}\rgroup
  {\! \! \!} w_{11} {\! \! \!} \\
  {\! \! \!} w_{12} {\! \! \!} \\
  {\! \! \!} w_{13} {\! \! \!}
\end{array}
   \! \! + i\mu \!
\begin{array} \lgroup{ccc}\rgroup
  {\! \! \!} 0 &    n_2    &    n_3 {\! \! \!} \\
  {\! \! \!} 0 &     0     &   0    {\! \! \!} \\
  {\! \! \!} 0 &     0     &   0    {\! \! \!}
\end{array} \! \! \! \!
\begin{array}\lgroup{c}\rgroup
  {\! \! \!} 0 {\! \! \!}  \\
  {\! \! \!} 0 {\! \! \!}  \\
  {\! \! \!} 1 {\! \! \!}
\end{array} {\! \!} =
\]

\medskip
\[
   {\! \! } = \mu {\! \! }
\begin{array}\lgroup{c}\rgroup
                     (n_2D_2+n_3D_3)w_{11}+in_3                  \\
   {\! \! \!} 2n_2D_2w_{12}+n_3D_3w_{12}+n_3D_2w_{13}  {\! \! \!} \\
   {\! \! \!} n_2D_3w_{12}+n_2D_2w_{13}+2n_3D_3w_{13} {\! \! \!}
\end{array}_{\!0} {\! \!}={\! \!}
\begin{array}\lgroup{c}\rgroup
  {\! \! \!} 0 {\! \! \!}  \\
  {\! \! \!} 0 {\! \! \!}  \\
  {\! \! \!} 0 {\! \! \!}
\end{array}.
\]
\\
  From this system of equations, we obtain that the associated vector $u_1$ has the form:
\[
  w_1=i(-x_3+k_3,0,0), \, \emph{where}\; k_3 \text{~---} \,
  \emph{constant depends on the domain}\: \mathcal{D}.
\]

  Let $w_2$~--- be the next associated eigenvector, then $w_2$ can be defined
  by the system of equations:

\[
  \mathcal{C}_0(0)w_2+\frac{\partial}{\partial\alpha}\mathcal{C}_0(0)w_1+
  \frac{1}{2}\frac{\partial^2}{\partial\alpha^2}\mathcal{C}_0(0)w_0=
\]
\[
  =-\mu \! \!
\begin{array} \lgroup{ccc}\rgroup
   {\! \! \!}  D^2_2+ D^2_3  &        0       &         0                 \\
                    0        &  2D^2_2+D^2_3  &      D_2D_3                \\
                    0        &     D_2D_3     &   D^2_2+2D^2_3  {\! \! \!}
\end{array} \! \! \! \!
\begin{array}\lgroup{c}\rgroup
  {\! \! \!}  w_{21} {\! \! \!}  \\
  {\! \! \!}  w_{22} {\! \! \!}  \\
  {\! \! \!}  w_{23} {\! \! \!}
\end{array} \!
  \! -i\mu \! \!
\begin{array}\lgroup{ccc}\rgroup
   {\! \! \!} 0   &   D_2   &   D_3 {\! \! \!}  \\
   {\! \! \!} D_2 &    0    &    0  {\! \! \!}   \\
   {\! \! \!} D_3 &    0    &    0  {\! \! \!}
\end{array} {\! \! \! \!}
\begin{array}\lgroup{c}\rgroup
   {\! \! \! \! \!}  -\!i(x_3-k_3)   {\! \! \!} \\
                             0                  \\
                             0
\end{array} \! +
\]

\medskip
\[
  +\, \mu {\! \!}
\begin{array}\lgroup{ccc}\rgroup
   {\! \! \!} 2  &   0  &   0  {\! \! \!} \\
   {\! \! \!} 0  &   1  &   0  {\! \! \!} \\
   {\! \! \!} 0  &   0  &   1  {\! \! \!}
\end{array} {\! \! \! \! \!}
\begin{array}\lgroup{c}\rgroup
  {\! \! \!} 0 {\! \! \!}  \\
  {\! \! \!} 0 {\! \! \!}  \\
  {\! \! \!} 1 {\! \! \!}
\end{array} {\! \!}
  =-\mu {\! \!}
\begin{array}\lgroup{c}\rgroup
                   \left(D^2_2+D^2_3\right)w_{21}    \\
   {\! \! \!} 2D^2_2w_{22}+D^2_3w_{22}+D_2D_3w_{23} {\! \! \!} \\
   {\! \! \!} D_2D_3w_{22}+D^2_2w_{23}+2D^2_3w_{23} {\! \! \!}
\end{array} {\! \! } = {\! \!}
\begin{array}\lgroup{c}\rgroup
  {\! \! \!} 0 {\! \! \!} \\
  {\! \! \!} 0 {\! \! \!} \\
  {\! \! \!} 0 {\! \! \!}
\end{array}\! ;
\]

\bigskip
\[
  \left.{\mathcal{U}(0)w_2+\frac{\partial}{\partial\alpha}\mathcal{U}(0)w_1}
  \right|_{\partial\mathcal{D}}=
\]
\[
  = \! \mu \!
\begin{array} \lgroup{ccc}\rgroup
  {\! \! \!} n_2D_2+n_3D_3    &           0        &       0            \\
         0          &   2n_2D_2+n_3D_3   &    n_3D_2            \\
         0          &        n_2D_3      & n_2D_2+2n_3D_3 {\! \! \!}
\end{array} \! \! \! \!
\begin{array}\lgroup{c}\rgroup
  {\! \! \!} w_{21} {\! \! \!} \\
  {\! \! \!} w_{22} {\! \! \!} \\
  {\! \! \!} w_{23} {\! \! \!}
\end{array}+
\]
\medskip
\[
    + \, i\mu \!
\begin{array} \lgroup{ccc}\rgroup
  {\! \! \!} 0 &    n_2    &    n_3 {\! \! \!} \\
  {\! \! \!} 0 &     0     &   0    {\! \! \!} \\
  {\! \! \!} 0 &     0     &   0    {\! \! \!}
\end{array} \! \! \! \!
\begin{array}\lgroup{c}\rgroup
   {\! \! \! \! \!}  -\!i(x_3-k_3)   {\! \! \!} \\
                             0                  \\
                             0
\end{array} \!
   {\! \! } = \mu {\! \! }
\begin{array}\lgroup{c}\rgroup
                         (n_2D_2+n_3D_3)w_{21}   \\
  {\! \! \!}   2n_2D_2w_{22}+n_3D_3w_{22}+n_3D_2w_{23}  {\! \! \!}  \\
  {\! \! \!}   n_2D_3w_{22}+n_2D_2w_{23}+2n_3D_3w_{23} {\! \! \!}
\end{array}_{\!0} {\! \!}={\! \!}
\begin{array}\lgroup{c}\rgroup
  {\! \! \!} 0 {\! \! \!}  \\
  {\! \! \!} 0 {\! \! \!}  \\
  {\! \! \!} 0 {\! \! \!}
\end{array}.
\]\\
  From this system of equations, we obtain that the associated vector $w_2$
  has the form:
\[
  w_2=(0,x_3,-x_2).
\]

  Let $w_3$~--- be the next associated eigenvector, then $w_3$ can be defined
  by the system of equations:

\[
  \mathcal{C}_0(0)w_3+\frac{\partial}{\partial\alpha}\mathcal{C}_0(0)w_2+
  \frac{1}{2}\frac{\partial^2}{\partial\alpha^2}\mathcal{C}_0(0)w_1=
\]
\[
  =-\mu \! \!
\begin{array} \lgroup{ccc}\rgroup
   {\! \! \!}  D^2_2+ D^2_3  &        0       &         0                 \\
                    0        &  2D^2_2+D^2_3  &      D_2D_3                \\
                    0        &     D_2D_3     &   D^2_2+2D^2_3  {\! \! \!}
\end{array} \! \! \! \!
\begin{array}\lgroup{c}\rgroup
  {\! \! \!}  w_{31} {\! \! \!}  \\
  {\! \! \!}  w_{32} {\! \! \!}  \\
  {\! \! \!}  w_{33} {\! \! \!}
\end{array} \!
  \! -i\mu \! \!
\begin{array}\lgroup{ccc}\rgroup
   {\! \! \!} 0   &   D_2   &   D_3 {\! \! \!}  \\
   {\! \! \!} D_2 &    0    &    0  {\! \! \!}   \\
   {\! \! \!} D_3 &    0    &    0  {\! \! \!}
\end{array} {\! \! \! \!}
\begin{array}\lgroup{c}\rgroup
  {\! \! \!}         0  {\! \! \!}  \\
  {\! \! \!}        x_3 {\! \! \!}  \\
  {\! \! \! \! \!} -x_2 {\! \! \!}
\end{array} {\! \!}+
\]
\smallskip
\[
  +\, \mu {\! \!}
\begin{array}\lgroup{ccc}\rgroup
   {\! \! \!} 2  &   0  &   0  {\! \! \!} \\
   {\! \! \!} 0  &   1  &   0  {\! \! \!} \\
   {\! \! \!} 0  &   0  &   1  {\! \! \!}
\end{array} {\! \! \! \! \!}
\begin{array}\lgroup{c}\rgroup
  {\! \! \! \! \!} -i(x_3-k_3) {\! \! \!}  \\
  {\! \! \!}         0     {\! \! \!}  \\
  {\! \! \!}         0     {\! \! \!}
\end{array} {\! \!}
  =-\mu {\! \!}
\begin{array}\lgroup{c}\rgroup
   {\! \! \! \!} \left(D^2_2+D^2_3\right)w_{31}+2i(x_3-k_3) {\! \! \! \!}   \\
   {\! \! \!} 2D^2_2w_{32}+D^2_3w_{32}+D_2D_3w_{33} {\! \! \!} \\
   {\! \! \!} D_2D_3w_{32}+D^2_2w_{33}+2D^2_3w_{33} {\! \! \!}
\end{array} {\! \! } = {\! \!}
\begin{array}\lgroup{c}\rgroup
  {\! \! \!} 0 {\! \! \!} \\
  {\! \! \!} 0 {\! \! \!} \\
  {\! \! \!} 0 {\! \! \!}
\end{array}\! ;
\]
\bigskip
\[
  \left.{\mathcal{U}(0)w_3+\frac{\partial}{\partial\alpha}\mathcal{U}(0)w_2}
  \right|_{\partial\mathcal{D}}=
\]
\[
  = \! \mu \!
\begin{array} \lgroup{ccc}\rgroup
  {\! \! \!} n_2D_2+n_3D_3    &           0        &       0            \\
         0          &   2n_2D_2+n_3D_3   &    n_3D_2            \\
         0          &        n_2D_3      & n_2D_2+2n_3D_3 {\! \! \!}
\end{array} \! \! \! \!
\begin{array}\lgroup{c}\rgroup
  {\! \! \!} w_{31} {\! \! \!} \\
  {\! \! \!} w_{32} {\! \! \!} \\
  {\! \! \!} w_{33} {\! \! \!}
\end{array}+
\]
\[
   \! \! + \, i\mu \!
\begin{array} \lgroup{ccc}\rgroup
  {\! \! \!} 0 &    n_2    &    n_3 {\! \! \!} \\
  {\! \! \!} 0 &     0     &   0    {\! \! \!} \\
  {\! \! \!} 0 &     0     &   0    {\! \! \!}
\end{array} \! \! \! \!
\begin{array}\lgroup{c}\rgroup
   {\! \! \! \! }  0   {\! \! \!} \\
                   x_3             \\
    {\! \! \! \! \! \! } -x_2   {\! \! \! \! }
\end{array} \! \!
   {\! \! } = \mu {\! \! }
\begin{array}\lgroup{c}\rgroup
  {\! \! \!}   (n_2D_2+n_3D_3)w_{31}+i(n_2x_3-n_3x_2) {\! \! \!}  \\
  {\! \! \!}   2n_2D_2w_{32}+n_3D_3w_{32}+n_3D_2w_{33}  {\! \! \!}  \\
  {\! \! \!}   n_2D_3w_{32}+n_2D_2w_{33}+2n_3D_3w_{33} {\! \! \!}
\end{array}_{\!0} {\! \!}={\! \!}
\begin{array}\lgroup{c}\rgroup
  {\! \! \!} 0 {\! \! \!}  \\
  {\! \! \!} 0 {\! \! \!}  \\
  {\! \! \!} 0 {\! \! \!}
\end{array}.
\]\\
  We will look for the solution of the system of equations in the form $w_3=(u_{31},0,0)$.
  From this system of equations, we examine the next boundary problem:
\begin{align} \label{n41}
      \left \{
        \begin{matrix}
          \qquad\mu\left(D^2_2+D^2_3\right)w_{31} & = & -2i\mu(x_3-k_3)
          \\
          \left.{\mu(n_2D_2+n_3D_3)w_{31}}\right|_{\partial\mathcal{D}} & = & -i\mu(n_2x_3-n_3x_2).
        \end{matrix}
       \right.
\end{align}
  Thus, in order to find $w_{31}$, one has to prove the solvability of the problem (\ref{n41}).
  This problem is solvable only, when we have the next equality:
\[
  2\,i \mu\!\!\int\limits_\mathcal{D}(x_3-k_3)\,dx+(-i\mu(n_2x_3-n_3x_2),1)_0=0\,.
\]
  We have that $(i\mu(n_2x_3-n_3x_2),1)_0=0$, see  (\ref{n33}), then
\[
  2\,i \mu\! \!\int\limits_\mathcal{D}(x_3-k_3)\,dx+(-i\mu(n_2x_3-n_3x_2),1)_0=
  2\,i \mu\! \!\int\limits_\mathcal{D}(x_3-k_3)\,dx.
\]
  Determining now the constant $k_3$ from the condition $\int\limits_\mathcal{D}(x_3-k_3)\,dx=0$,
  we find that the boundary problem (\ref{n41}) is solvable and thus, one can take vector
  $w_3=(w_{31},0,0)$, where $w_{31}$ is some solution of the problem (\ref{n41}).\\

  Let $w_4$~--- be the next associated eigenvector, then $w_4$ can be defined
  by the system of equations:

\[
  \mathcal{C}_0(0)w_4+\frac{\partial}{\partial\alpha}\mathcal{C}_0(0)w_3+
  \frac{1}{2}\frac{\partial^2}{\partial\alpha^2}\mathcal{C}_0(0)w_2=
\]
\[
  =-\mu \! \!
\begin{array} \lgroup{ccc}\rgroup
   {\! \! \!}  D^2_2+ D^2_3  &        0       &         0                 \\
                    0        &  2D^2_2+D^2_3  &      D_2D_3                \\
                    0        &     D_2D_3     &   D^2_2+2D^2_3  {\! \! \!}
\end{array} \! \! \! \!
\begin{array}\lgroup{c}\rgroup
  {\! \! \!}  w_{41} {\! \! \!}  \\
  {\! \! \!}  w_{42} {\! \! \!}  \\
  {\! \! \!}  w_{43} {\! \! \!}
\end{array} \!
  \! -i\mu \! \!
\begin{array}\lgroup{ccc}\rgroup
   {\! \! \!} 0   &   D_2   &   D_3 {\! \! \!}  \\
   {\! \! \!} D_2 &    0    &    0  {\! \! \!}   \\
   {\! \! \!} D_3 &    0    &    0  {\! \! \!}
\end{array} {\! \! \! \!}
\begin{array}\lgroup{c}\rgroup
  {\! \! \!}    w_{31}  {\! \! \!}  \\
  {\! \! \!}      0    {\! \! \!}  \\
  {\! \! \!}      0   {\! \! \!}
\end{array} {\! \!}+
\]
\smallskip
\[
  +\, \mu {\! \!}
\begin{array}\lgroup{ccc}\rgroup
   {\! \! \!} 2  &   0  &   0  {\! \! \!} \\
   {\! \! \!} 0  &   1  &   0  {\! \! \!} \\
   {\! \! \!} 0  &   0  &   1  {\! \! \!}
\end{array} {\! \! \! \! \!}
\begin{array}\lgroup{c}\rgroup
   {\! \! \! \! }  0   {\! \! \!} \\
                   x_3             \\
    {\! \! \! \! \! \! } -x_2   {\! \! \! \! }
\end{array} \! \!
 =-\mu {\! \!}
\begin{array}\lgroup{c}\rgroup
   {\! \! \! \!} \left(D^2_2+D^2_3\right)w_{41} {\! \! \! \!}   \\
   {\! \! \!} 2D^2_2w_{42}+D^2_3w_{42}+D_2D_3w_{43}+iD_2w_{31}-x_3 {\! \! \!} \\
   {\! \! \!} D_2D_3w_{42}+D^2_2w_{43}+2D^2_3w_{43}+iD_3w_{31}+x_2 {\! \! \!}
\end{array} {\! \! } = {\! \!}
\begin{array}\lgroup{c}\rgroup
  {\! \! \!} 0 {\! \! \!} \\
  {\! \! \!} 0 {\! \! \!} \\
  {\! \! \!} 0 {\! \! \!}
\end{array}\! ;
\]

\bigskip
\[
  \left.{\mathcal{U}(0)w_4+\frac{\partial}{\partial\alpha}\mathcal{U}(0)w_3}
  \right|_{\partial\mathcal{D}}=
\]
\[
  = \! \mu \!
\begin{array} \lgroup{ccc}\rgroup
  {\! \! \!} n_2D_2+n_3D_3    &           0        &       0            \\
         0          &   2n_2D_2+n_3D_3   &    n_3D_2            \\
         0          &        n_2D_3      & n_2D_2+2n_3D_3 {\! \! \!}
\end{array} \! \! \! \!
\begin{array}\lgroup{c}\rgroup
  {\! \! \!} w_{41} {\! \! \!} \\
  {\! \! \!} w_{42} {\! \! \!} \\
  {\! \! \!} w_{43} {\! \! \!}
\end{array}
   \! \! + \, i\mu \!
\begin{array} \lgroup{ccc}\rgroup
  {\! \! \!} 0 &    n_2    &    n_3 {\! \! \!} \\
  {\! \! \!} 0 &     0     &   0    {\! \! \!} \\
  {\! \! \!} 0 &     0     &   0    {\! \! \!}
\end{array} \! \! \! \!
\begin{array}\lgroup{c}\rgroup
  {\! \! \!}    w_{31}  {\! \! \!}  \\
  {\! \! \!}      0    {\! \! \!}  \\
  {\! \! \!}      0   {\! \! \!}
\end{array} {\! \!}=
\]
\medskip
\[
  = \mu
\begin{array}\lgroup{c}\rgroup
  {\! \! \!}   (n_2D_2+n_3D_3)w_{41} {\! \! \!}  \\
  {\! \! \!}   2n_2D_2w_{42}+n_3D_3w_{42}+n_3D_2w_{43}  {\! \! \!}  \\
  {\! \! \!}   n_2D_3w_{42}+n_2D_2w_{43}+2n_3D_3w_{43} {\! \! \!}
\end{array}_{\!0} {\! \!}={\! \!}
\begin{array}\lgroup{c}\rgroup
  {\! \! \!} 0 {\! \! \!}  \\
  {\! \! \!} 0 {\! \! \!}  \\
  {\! \! \!} 0 {\! \! \!}
\end{array}.
\]
  The system of equation have solution $w_4$ if the next boundary problem
  is solvable:
\begin{align} \label{n42}
\begin{matrix}
 & - \mu {\! \!}
  \begin{array}\lgroup{cc}\rgroup
     {\! \! \!} 2D^2_2+D^2_3   &     D_2D_3     \\
         D_2D_3             &  D^2_2+2D^2_3 {\! \! \!}
  \end{array} {\! \! \! \! \!}
  \begin{array}\lgroup{c}\rgroup
  {\! \! \!} w_{42} {\! \! \!} \\
  {\! \! \!} w_{43} {\! \! \!}
  \end{array}  {\! \! \!} =
  i\mu {\! \!}
  \begin{array}\lgroup{cc}\rgroup
  {\! \! \!} D_2  &  0  {\! \! \!}  \\
  {\! \! \!} D_3  &  0  {\! \! \!}
  \end{array} {\! \! \! \! \!}
  \begin{array}\lgroup{c}\rgroup
  {\! \! \!} w_{31} {\! \! \!} \\
  {\! \! \!}   0    {\! \! \!}
  \end{array}  {\! \! \!}
  -\mu {\! \!}
  \begin{array}\lgroup{cc}\rgroup
  {\! \! \!} 1  &  0  {\! \! \!}  \\
  {\! \! \!} 0  &  1  {\! \! \!}
  \end{array} {\! \! \! \! \!}
  \begin{array}\lgroup{c}\rgroup
  {\! \! \!}   x_3 {\! \! \!} \\
  {\! \! \! \! \!}  -x_2    {\! \! \!}
  \end{array}=f
\\
\smallskip
\\
 & \mu {\! \!}
  \begin{array}\lgroup{cc}\rgroup
     {\! \! \!} 2n_2D_2+n_3D_3   &     n_3D_2     \\
              n_2D_3             &  n_2D_2+2n_3D_3 {\! \! \!}
  \end{array} {\! \! \! \! \!}
  \begin{array}\lgroup{c}\rgroup
  {\! \! \!} w_{42} {\! \! \!} \\
  {\! \! \!} w_{43}    {\! \! \!}
  \end{array}_{\partial\mathcal{D}}  {\! \! \!} =
  -\,i  {\! \!}
  \begin{array}\lgroup{cc}\rgroup
  {\! \! \!} 0  &  0  {\! \! \!}  \\
  {\! \! \!} 0  &  0  {\! \! \!}
  \end{array} {\! \! \! \! \!}
  \begin{array}\lgroup{c}\rgroup
  {\! \! \!} w_{31} {\! \! \!} \\
  {\! \! \!}   0    {\! \! \!}
  \end{array} = h
\end{matrix}
\end{align}\\
  Therefore, the boundary problem (\ref{n42}) is solvable if and
  only  if we have the relations:
\[
  (f,e_k)+(h,e_k)_0=0, \, k=1,2,3,
\]
  where $e_1=(1,0)$, $e_2=(0,1)$ and $e_3=(x_3,-x_2)$ are solutions of the homogeneous boundary
  problem (\ref{n42}), i.e., when $f=h=0$. We examine if this case is true for the vector
  $e_2=(0,1)$, then
\[
  (f,e_2)+(h,e_2)_0=i\mu(D_3w_{31},1)+\mu(x_2,1)=i\mu(D_3w_{31},D_3x_3)+\mu(x_2,D_3x_3)=
\]
\[
  =-i\mu(D^2_3w_{31},x_3)+\mu(in_3D_3w_{31},x_3)_0+\mu(D_3x_2,x_3)+\mu(n_3x_2,x_3)_0=
\smallskip
\]
\[
  =i\mu(D^2_2w_{31},x_3)-2\mu(x_3-k_3,x_3)+\mu(in_3D_3w_{31}+n_3x_2,x_3)_0=i\mu(D^2_2w_{31},x_3)-
\smallskip
\]
\[
  -\mu(in_2D_2w_{31},x_3)_0+\mu(n_2x_3,x_3)_0-2\mu(x_3-k_3,x_3)=i\mu(D^2_2w_{31},x_3)
  -i\mu(D^2_2w_{31},x_3)+
\smallskip
\]
\[
  +\mu(D_2w_{31},iD_2x_3)+\mu(D_2x_3,x_3)+\mu(x_3,D_2x_3)-2\mu(x_3-k_3,x_3)=
\smallskip
\]
\[
  -2\mu(x_3-k_3,x_3)=-2\mu\|x_3-k_3\|^2<0.
\smallskip
\]
  For this inequality the boundary problem (\ref{n42})
  and at the same time the system of equations for the
  vector are not solvable. Consequently, the chain of the vectors:
\[
  w_0\!=\!(0,0,1), \, w_1\!=\!-i(x_3-k_3,0,0), \, w_2\!=\!(0,x_3,-x_2) \; \emph{and}
  \; w_3\!=\!(w_{31},0,0)\text{~---} \, \emph{\,is maximal}.
\]\\
  d) Let $g_1$~--- be an associated vector to be eigenvector $g_ 0$,
  then $g_1$ can be defined by the system of equations:

\[
  \mathcal{C}_0(0)g_1+\frac{\partial}{\partial\alpha}\mathcal{C}_0(0)g_0=
\]
\[
  =-\mu \! \!
\begin{array} \lgroup{ccc}\rgroup
   {\! \! \!}  D^2_2+ D^2_3  &        0       &         0                 \\
                    0        &  2D^2_2+D^2_3  &      D_2D_3                \\
                    0        &     D_2D_3     &   D^2_2+2D^2_3  {\! \! \!}
\end{array} \! \! \! \!
\begin{array}\lgroup{c}\rgroup
  {\! \! \!}  g_{11} {\! \! \!}  \\
  {\! \! \!}  g_{12} {\! \! \!}  \\
  {\! \! \!}  g_{13} {\! \! \!}
\end{array} \!
  \! -i\mu \! \!
\begin{array}\lgroup{ccc}\rgroup
   {\! \! \!} 0   &   D_2   &   D_3 {\! \! \!}  \\
   {\! \! \!} D_2 &    0    &    0  {\! \! \!}   \\
   {\! \! \!} D_3 &    0    &    0  {\! \! \!}
\end{array} \! \! \! \!
\begin{array}\lgroup{c}\rgroup
   {\! \! \!} 0 {\! \! \!} \\
   {\! \! \!} x_3 {\! \! \!} \\
   {\! \! \! \! \! \!} -x_2  {\! \! \!}
\end{array} \! =
\]
\medskip
\[
  =-\mu \!
\begin{array}\lgroup{c}\rgroup
  {\! \! \!}  \left(D^2_2+D^2_3\right)g_{11} {\! \! \!} \\
   {\! \! \!} 2D^2_2g_{12}+D^2_3g_{12}+D_2D_3g_{13} {\! \! \!} \\
   {\! \! \!} D_2D_3g_{12}+D^2_2g_{13}+2D^2_3g_{13} {\! \! \!}
\end{array} {\! \! } = \!
\begin{array}\lgroup{c}\rgroup
  {\! \! \!} 0 {\! \! \!} \\
  {\! \! \!} 0 {\! \! \!} \\
  {\! \! \!} 0 {\! \! \!}
\end{array};
\]

\bigskip
\[
  \left.{\mathcal{U}(0)g_1+\frac{\partial}{\partial\alpha}\mathcal{U}(0)g_0}
  \right |_{\partial\mathcal{D}}=
\]
\[
  = \! \mu \!
\begin{array} \lgroup{ccc}\rgroup
  {\! \! \!} n_2D_2+n_3D_3    &           0        &       0            \\
         0          &   2n_2D_2+n_3D_3   &    n_3D_2            \\
         0          &        n_2D_3      & n_2D_2+2n_3D_3 {\! \! \!}
\end{array} \! \! \! \!
\begin{array}\lgroup{c}\rgroup
  {\! \! \!} g_{11} {\! \! \!} \\
  {\! \! \!} g_{12} {\! \! \!} \\
  {\! \! \!} g_{13} {\! \! \!}
\end{array}
   \! \! + i\mu \!
\begin{array} \lgroup{ccc}\rgroup
  {\! \! \!} 0 &    n_2    &    n_3 {\! \! \!} \\
  {\! \! \!} 0 &     0     &   0    {\! \! \!} \\
  {\! \! \!} 0 &     0     &   0    {\! \! \!}
\end{array} \! \! \! \!
\begin{array}\lgroup{c}\rgroup
  {\! \! \!} 0 {\! \! \!}  \\
  {\! \! \!} x_3 {\! \! \!}  \\
  {\! \! \! \! \!} -x_2 {\! \! \!}
\end{array} {\! \!} =
\]
\medskip
\[
   {\! \! } = \mu {\! \! }
\begin{array}\lgroup{c}\rgroup
    {\! \! \!} (n_2D_2+n_3D_3)g_{11}+in_2x_3-in_3x_2  {\! \! \!} \\
   {\! \! \!} 2n_2D_2g_{12}+n_3D_3g_{12}+n_3D_2g_{13}  {\! \! \!} \\
   {\! \! \!} n_2D_3g_{12}+n_2D_2g_{13}+2n_3D_3g_{13} {\! \! \!}
\end{array}_{\!0} {\! \!}={\! \!}
\begin{array}\lgroup{c}\rgroup
  {\! \! \!} 0 {\! \! \!}  \\
  {\! \! \!} 0 {\! \! \!}  \\
  {\! \! \!} 0 {\! \! \!}
\end{array}.
\]
  We will look for the solution of the system of equations in the form $g_1=(g_{11},0,0)$.
  From this system of equations, we examine the next boundary problem:
\begin{align} \label{n43}
      \left \{
        \begin{matrix}
          \qquad\mu\left(D^2_2+D^2_3\right)g_{11} & = & 0
          \\
          \left.{\mu(n_2D_2+n_3D_3)g_{11}}\right|_{\partial\mathcal{D}} & = & -i\mu(n_2x_3-n_3x_2).
        \end{matrix}
       \right.
\end{align}
  Thus, in order to find $g_{11}$, one has to prove the solvability of the problem (\ref{n43}).
  This problem is solvable only, when we have the next equality:
\[
  -i\mu(n_2x_3-n_3x_2,1)_0=0.
\]
  In fact,
\[
  -i\mu(n_2x_3-n_3x_2,1)_0=-\mu(in_2x_3,1)_0+\mu(in_3x_2,1)_0=
\]
\[
  =\mu(x_3,iD_21)-\mu(iD_2x_3,1)-\mu(x_2,iD_31)+\mu(iD_3x_2,1)=0.
\smallskip
\]
  We find that the boundary problem (\ref{n43}) is solvable and thus, one can take vector
  $g_1=(g_{11},0,0)$, where $g_{11}$ is some solution of the problem
  (\ref{n43}).\\

  Let $g_2$~--- be the next associated eigenvector, then $g_2$ can be defined
  by the system of equations:
\[
  \mathcal{C}_0(0)g_2+\frac{\partial}{\partial\alpha}\mathcal{C}_0(0)g_1+
  \frac{1}{2}\frac{\partial^2}{\partial\alpha^2}\mathcal{C}_0(0)g_0=
\]
\[
  =-\mu \! \!
\begin{array} \lgroup{ccc}\rgroup
   {\! \! \!}  D^2_2+ D^2_3  &        0       &         0                 \\
                    0        &  2D^2_2+D^2_3  &      D_2D_3                \\
                    0        &     D_2D_3     &   D^2_2+2D^2_3  {\! \! \!}
\end{array} \! \! \! \!
\begin{array}\lgroup{c}\rgroup
  {\! \! \!}  g_{21} {\! \! \!}  \\
  {\! \! \!}  g_{22} {\! \! \!}  \\
  {\! \! \!}  g_{23} {\! \! \!}
\end{array} \!
  \! -i\mu \! \!
\begin{array}\lgroup{ccc}\rgroup
   {\! \! \!} 0   &   D_2   &   D_3 {\! \! \!}  \\
   {\! \! \!} D_2 &    0    &    0  {\! \! \!}   \\
   {\! \! \!} D_3 &    0    &    0  {\! \! \!}
\end{array} {\! \! \! \!}
\begin{array}\lgroup{c}\rgroup
   {\! \! \! \! \!}  g_{11}   {\! \! \! \! \!} \\
                       0                  \\
                       0
\end{array} \! +
\]

\medskip
\[
  +\, \mu {\! \!}
\begin{array}\lgroup{ccc}\rgroup
   {\! \! \!} 2  &   0  &   0  {\! \! \!} \\
   {\! \! \!} 0  &   1  &   0  {\! \! \!} \\
   {\! \! \!} 0  &   0  &   1  {\! \! \!}
\end{array} {\! \! \! \! \!}
\begin{array}\lgroup{c}\rgroup
   {\! \! \! \! }  0   {\! \! \!} \\
                   x_3             \\
    {\! \! \! \! \! \! } -x_2   {\! \! \! \! }
\end{array} \! \!
 =-\mu {\! \!}
\begin{array}\lgroup{c}\rgroup
   {\! \! \! \!} \left(D^2_2+D^2_3\right)g_{21} {\! \! \! \!}   \\
   {\! \! \!} 2D^2_2g_{22}+D^2_3g_{22}+D_2D_3g_{23}+iD_2g_{11}-x_3 {\! \! \!} \\
   {\! \! \!} D_2D_3g_{22}+D^2_2g_{23}+2D^2_3g_{23}+iD_3g_{11}+x_2 {\! \! \!}
\end{array} {\! \! } = {\! \!}
\begin{array}\lgroup{c}\rgroup
  {\! \! \!} 0 {\! \! \!} \\
  {\! \! \!} 0 {\! \! \!} \\
  {\! \! \!} 0 {\! \! \!}
\end{array}\! ;
\]

\bigskip
\[
  \left.{\mathcal{U}(0)g_2+\frac{\partial}{\partial\alpha}\mathcal{U}(0)g_1}
  \right|_{\partial\mathcal{D}}=
\]
\[
  = \! \mu \!
\begin{array} \lgroup{ccc}\rgroup
  {\! \! \!} n_2D_2+n_3D_3    &           0        &       0            \\
         0          &   2n_2D_2+n_3D_3   &    n_3D_2            \\
         0          &        n_2D_3      & n_2D_2+2n_3D_3 {\! \! \!}
\end{array} \! \! \! \!
\begin{array}\lgroup{c}\rgroup
  {\! \! \!} g_{21} {\! \! \!} \\
  {\! \! \!} g_{22} {\! \! \!} \\
  {\! \! \!} g_{23} {\! \! \!}
\end{array}
   \! \! + \, i\mu \!
\begin{array} \lgroup{ccc}\rgroup
  {\! \! \!} 0 &    n_2    &    n_3 {\! \! \!} \\
  {\! \! \!} 0 &     0     &   0    {\! \! \!} \\
  {\! \! \!} 0 &     0     &   0    {\! \! \!}
\end{array} \! \! \! \!
\begin{array}\lgroup{c}\rgroup
  {\! \! \!}    g_{11}  {\! \! \!}  \\
  {\! \! \!}      0    {\! \! \!}  \\
  {\! \! \!}      0   {\! \! \!}
\end{array} {\! \!}=
\]
\medskip
\[
  = \mu
\begin{array}\lgroup{c}\rgroup
  {\! \! \!}   (n_2D_2+n_3D_3)g_{21} {\! \! \!}  \\
  {\! \! \!}   2n_2D_2g_{22}+n_3D_3g_{22}+n_3D_2g_{23}  {\! \! \!}  \\
  {\! \! \!}   n_2D_3g_{22}+n_2D_2g_{23}+2n_3D_3g_{23} {\! \! \!}
\end{array}_{\!0} {\! \!}={\! \!}
\begin{array}\lgroup{c}\rgroup
  {\! \! \!} 0 {\! \! \!}  \\
  {\! \! \!} 0 {\! \! \!}  \\
  {\! \! \!} 0 {\! \! \!}
\end{array}.
\]
  The system of equation have solution $g_2$ if the next boundary problem
  is solvable:
\begin{align} \label{n44}
\begin{matrix}
 & - \mu {\! \!}
  \begin{array}\lgroup{cc}\rgroup
     {\! \! \!} 2D^2_2+D^2_3   &     D_2D_3     \\
         D_2D_3             &  D^2_2+2D^2_3 {\! \! \!}
  \end{array} {\! \! \! \! \!}
  \begin{array}\lgroup{c}\rgroup
  {\! \! \!} g_{22} {\! \! \!} \\
  {\! \! \!} g_{23} {\! \! \!}
  \end{array}  {\! \! \!} =
  i\mu {\! \!}
  \begin{array}\lgroup{cc}\rgroup
  {\! \! \!} D_2  &  0  {\! \! \!}  \\
  {\! \! \!} D_3  &  0  {\! \! \!}
  \end{array} {\! \! \! \! \!}
  \begin{array}\lgroup{c}\rgroup
  {\! \! \!} g_{11} {\! \! \!} \\
  {\! \! \!}   0    {\! \! \!}
  \end{array}  {\! \! \!}
  -\mu {\! \!}
  \begin{array}\lgroup{cc}\rgroup
  {\! \! \!} 1  &  0  {\! \! \!}  \\
  {\! \! \!} 0  &  1  {\! \! \!}
  \end{array} {\! \! \! \! \!}
  \begin{array}\lgroup{c}\rgroup
  {\! \! \!}   x_3 {\! \! \!} \\
  {\! \! \! \! \!}  -x_2    {\! \! \!}
  \end{array}=f
\\
\smallskip
\\
 & \mu {\! \!}
  \begin{array}\lgroup{cc}\rgroup
     {\! \! \!} 2n_2D_2+n_3D_3   &     n_3D_2     \\
              n_2D_3             &  n_2D_2+2n_3D_3 {\! \! \!}
  \end{array} {\! \! \! \! \!}
  \begin{array}\lgroup{c}\rgroup
  {\! \! \!} g_{22} {\! \! \!} \\
  {\! \! \!} g_{23}    {\! \! \!}
  \end{array}_{\partial\mathcal{D}}  {\! \! \!} =
  -\,i  {\! \!}
  \begin{array}\lgroup{cc}\rgroup
  {\! \! \!} 0  &  0  {\! \! \!}  \\
  {\! \! \!} 0  &  0  {\! \! \!}
  \end{array} {\! \! \! \! \!}
  \begin{array}\lgroup{c}\rgroup
  {\! \! \!} g_{11} {\! \! \!} \\
  {\! \! \!}   0    {\! \! \!}
  \end{array} = h
\end{matrix}
\end{align}\\
  Therefore, the boundary problem (\ref{n44}) is solvable if and
  only  if we have the relations:
\[
  (f,e_k)+(h,e_k)_0=0, \, k=1,2,3,
\]
  where $e_1=(1,0)$, $e_2=(0,1)$ and  $e_3=(x_3,-x_2)$ is the solution of the
  homogeneous boundary problem (\ref{n44}), i.e., when $f=h=0$. We examine if
  this case is true for the vector $e_3=(x_3,-x_2)$, then
\[
  (f,e_3)+(h,e_3)_0=i\mu(D_2g_{11},x_3)-i\mu(D_3g_{11},x_2)-\mu\|(x_3,-x_2)\|^2=
\]
\[
  =-\mu(g_{11},in_2x_3)_0+\mu(g_{11},in_3x_2)_0-\mu\|(x_3,-x_2)\|^2=
\smallskip
\]
\[
  =\mu(g_{11},-i(n_2x_3-n_3x_2))_0-\mu\|(x_3,-x_2)\|^2=\mu(g_{11},(n_2D_2+n_3D_3)g_{11})_0\,-
\smallskip
\]
\[
  -\mu\|(x_3,-x_2)\|^2=\mu(g_{11},n_2D_2g_{11})+\mu(g_{11},n_3D_3g_{11})-\mu\|(x_3,-x_2)\|^2=
\smallskip
\]
\[
  =\mu(D_2g_{11},D_2g_{11})+\mu(D_3g_{11},D_3g_{11})+\mu(g_{11},(D^2_2+D^2_3)g_{11})-\mu\|(x_3,-x_2)\|^2=
\smallskip
\]
\[
  =\mu\left[\,\|D_2g_{11}\|^2+\|D_3g_{11}\|^2-\|(x_3,-x_2)\|^2\right].
\smallskip
\]
  On the other hand, since the vectors $(D_2g_{11},D_3g_{11})$ and
  $(x_3,-x_2)$  are not proportional, by the Cauchy -- Bunyakovskii inequality we have:
\[
  \|D_2g_{11}\|^2+\|D_3g_{11}\|^2=i\langle(D_2g_{11},D_3g_{11}),(x_3,-x_2)\rangle<
  \|(D_2g_{11},D_3g_{11}\|\|(x_3,-x_2)\|=
\]
\[
  =\left[\,\|D_2g_{11}\|^2+\|D_3g_{11}\|^2\right]^{\frac{1}{2}}\|(x_3,-x_2)\|,
\]
  whence
\[
  \left[\,\|D_2g_{11}\|^2+\|D_3g_{11}\|^2\right]^{\frac{1}{2}}<\|(x_3,-x_2)\|\, \Rightarrow
  \, \mu\left[\,\|D_2g_{11}\|^2+\|D_3g_{11}\|^2-\|(x_2,-x_2)\|^2\right]<0.
\smallskip
\]
  For this inequality the boundary problem (\ref{n44})
  and at the same time the system of equations for the
  vector are not solvable. Consequently, the chain of the vectors:
\[
  \qquad g_0=(0,x_3,-x_2) \;\; \emph{and} \;\; g_1=(g_{11},0,0)\text{~---}\,\emph{is maximal}.
\]

  The theorem is proved.

\bigskip
  According to theorems~\ref{T2} and~\ref{T3} in the static case, i.e., when $\omega=0$,
  for the isotropic structure and when the Lame constants are $\mu>0$ and $\lambda=0$,
  we obtain that if to all eigen-and associated vectors of the
  spectral problem $\mathcal{L}_0(\alpha)$, corresponding to the eigenvalue $\alpha_n$
  from the upper semiplane, one adjoint the vectors
  $v_0, \: u_0, \: u_1, \: w_0, \: w_1 \; \text{and} \; \, g_0$  from (\ref{n37}),
  then we obtain a complete and minimal system in the space $W^1_2(\mathcal{D})
  \; (\text{and in} \: L_2(\mathcal{D}))$.\\

\textbf{6.~Theorems of motion for the real eigenvalues}

\bigskip
  The pencils $L_\omega(\alpha)$ and $L^0_\omega(\alpha)$ are particular cases of
  the pencil $L(\alpha,\theta)$, see \cite{KAGyOMB81}. Therefore, taking into account
  theorem~\ref{T3} and theorem 3.3 (the case A.1) \cite{KAGyOMB81},
  recalling corollary 2.4 \cite{KAGyOMB81}, we obtain:

\begin{theorem} \label{T4}
  In a sufficiently small neighborhood of the point $\alpha=0$ for
  $0<\omega^2<\varepsilon, \; \varepsilon$~--- is small the spectral problem
  $\mathcal{L}_\omega(\alpha)$ has four pairs
  $(\alpha^{\pm}_j(\omega^2),\psi_j^\pm(\omega^2)), \: j=1,\dots,4,$
  of the firs and the second kind, which move at increase of
  $0<\omega^2<\varepsilon$ to the right and to the left,
  respectively, of the point $\alpha=0$. In addition, one has two more
  pairs $(\alpha^k(\omega^2),\psi^k(\omega^2)), k=\pm1,\pm2,$ which shift
  into the complex plane along the curves $\mu^k$, tangent to the imaginary axis.
\end{theorem}

  The total algebraic multiplicity of the real eigenvalues of the
  spectral problems $\mathcal{L}_\omega(\alpha)$ and $\mathcal{L}^0_\omega(\alpha)$
  admites the next estimates, see \cite{KAGyOMB81},

\begin{align} \label{n45}
      \sum_{\alpha\in{\sigma(\mathcal{L}_\omega)\cap\textbf{R}}}\chi(\alpha)
      \geq 2N^0_{\mathcal{L}_\omega}(\omega^2)\geq8;
\end{align}
\begin{align} \label{n46}
      \sum_{\alpha\in{\sigma(\mathcal{L}^0_\omega)\cap\textbf{R}}}\chi(\alpha)
      \geq 2N^0_{\mathcal{L}^0_\omega}(\omega^2). \! \! \qquad
\end{align}
  At the subsequent increase of the parameter $\omega^2>0$ the
  qualitative picture of the motion of the real eigenvalues of these
  problems take place according to the next theorem:

\begin{theorem} \label{T5}
  There exist a sequence of numbers \,
  $\omega^2_1<\omega^2_2<\dots<\omega^2_n<\dots,\;\omega^2_n\rightarrow\infty,\;\omega_1=0$,
  for which the problem $\mathcal{L}_{\omega_k}(\alpha)$ has  on the real axis
  at least one neutral eigenvalue $\alpha_0$ (i.e., such that there corresponds to it an eigenfunction
  $y_0$, having associated functions). In all the remaining points $\omega^2\neq\omega^2_k$
  the problem $\mathcal{L}_\omega(\alpha)$ has on the real axis only pairs of the firs and the second
  kind, moving to the right and to the left, respectively, at the increase of
  $\omega^2\in(\omega^2_k,\omega^2_{k+1}), k=1,2,\dots$\:.
\end{theorem}

  A similar assertion holds for
  $\mathcal{L}^0_\omega(\alpha)$.\\

\bigskip
\textbf{7.~Spectral problems
  $\boldsymbol{\mathcal{L}_\omega(\alpha)}$ and
  $\boldsymbol{\mathcal{L}^0_\omega(\alpha)}$ in the semi-band}

\bigskip
  We examine the equations of the steady-state oscillations in the semiband
  $\prod=[\,0,1]\times\textbf{R}^+_{x_1}$ where
  $[\,0,1]\subset\textbf{R}_{x_2}$, with a free (or fixed) boundary
  in the isotropic structure and when the Lame constant are $\mu>0$ and $\lambda=0$,
  then, separating the variables $u(x_1,x_2)=e^{i\alpha x_1}v(x_2)$,
  as in the case of the semi-cylinder $\Omega$, we arrive at the spectral problem
  $\mathcal{L}_\omega(\alpha)$ (or $\mathcal{L}^0_\omega(\alpha)$)
  of the form (\ref{n17})--(\ref{n18}), except that the now:
\[
  \mathcal{A}=\mu \!
\begin{array}\lgroup{cc}\rgroup
    {\! \! \!} 2  &  0 {\! \! \!}  \\
    {\! \! \!} 0  &  1 {\! \! \!}
\end{array} {\! \!} , \;
  \mathcal{B}=i\mu \!
\begin{array}\lgroup{cc}\rgroup
    {\! \! \!} 0    &  D_2 {\! \! \!}  \\
    {\! \! \!} D_2  &  0   {\! \! \!}
\end{array} {\! \!} , \;
  \mathcal{C}=-\mu \!
\begin{array}\lgroup{cc}\rgroup
    {\! \! \!} D^2_2   &    0     {\! \! \!}  \\
    {\! \! \!}   0     &  D^2_2   {\! \! \!}
\end{array} {\! \!} , \;
  \mathcal{R}=
\begin{array}\lgroup{cc}\rgroup
    {\! \! \!} \rho   &    0   {\! \! \!}  \\
    {\! \! \!}   0    &  \rho  {\! \! \!}
\end{array} {\! \!} , \;
\]
\[
  \mathcal{M}=\mu \!
\begin{array}\lgroup{cc}\rgroup
    {\! \! \!}  n_2D_2   &    0       {\! \! \!}  \\
    {\! \! \!}     0     &  2n_2D_2   {\! \! \!}
\end{array} {\! \!} , \;
  \mathcal{B}=\mu \!
\begin{array}\lgroup{cc}\rgroup
    {\! \! \!}  0    &  n_2 {\! \! \!}  \\
    {\! \! \!}  0    &   0   {\! \! \!}
\end{array} {\! \!} , \;
\]

\[
  \text{where}\;\, n_2=\pm 1,\, v=(v_1,v_2),\, D_2=\frac{\partial}{\partial\alpha^2}\,.
\]
  With the corresponding simplifications, all the previous results
  are preserved for the spectral problems
  $\mathcal{L}_\omega(\alpha)$ and $\mathcal{L}^0_\omega(\alpha)$.
  For example, one has

\begin{theorem} \label{T6}
  For the isotropic structure, when the Lame constants are $\mu>0$ and
  $\lambda=0$, the spectral problem $\mathcal{L}_0(\alpha)$,
  at the point $\alpha=0$, has only two chains of eigen-and associated vectors:
  \begin{align} \label{n47}
    \begin{matrix}
    &  v_0=(1,0), &  v_1=(0,1);  &\\
\\
    u_0=(0,1), &  u_1=-i(x_2-\frac{1}{2},0), &  u_2=(1,0), &  u_3=-i(\frac{1}{3}x^3_2-\frac{1}{2}x^2_2,0).\\
    \end{matrix}
  \end{align}
\end{theorem}

\textbf{Proof.}\\
  a) Let $v_1$~--- be an associated vector to be eigenvector $v_ 0$,
  then $v_1$ can be defined by the system of equations:
\[
  \mathcal{C}_0(0)v_1+\frac{\partial}{\partial\alpha}\mathcal{C}_0(0)v_0=
\]
\[
  =-\mu \!
  \begin{array}\lgroup{cc}\rgroup
  {\! \! \!}  D^2_2   &    0     {\! \! \!}  \\
  {\! \! \!}     0    &  2D^2_2  {\! \! \!}
  \end{array}{\! \! \! \! \!}
  \begin{array}\lgroup{c}\rgroup
  {\! \! \!}  v_{11}  {\! \! \!} \\
  {\! \! \!}  v_{12}  {\! \! \!}
  \end{array} \!
  -i\mu \!
  \begin{array}\lgroup{cc}\rgroup
  {\! \! \!}   0   &   D_2  {\! \! \!} \\
  {\! \! \!}  D_2  &    0   {\! \! \!}
  \end{array} {\! \! \! \! \!}
  \begin{array}\lgroup{c}\rgroup
  {\! \! \!}  1  {\! \! \!} \\
  {\! \! \!}  0  {\! \! \!}
  \end{array}= -\mu
  \begin{array}\lgroup{c}\rgroup
  {\! \! \!}   D^2_2v_{11}  {\! \! \!} \\
  {\! \! \!}  2D^2_2v_{12}  {\! \! \!}
  \end{array} \!=
  \begin{array}\lgroup{c}\rgroup
  {\! \! \!}  0  {\! \! \!} \\
  {\! \! \!}  0  {\! \! \!}
  \end{array};
\]

\bigskip
\[
  \left.\mathcal{U}(0)v_1+\frac{\partial}{\partial\alpha}\mathcal{U}(0)v_0\right|^1_0=
\]
\smallskip
\[
  =\mu \!
  \begin{array}\lgroup{cc}\rgroup
  {\! \! \!}  n_2D_2   &   0     {\! \! \!}  \\
  {\! \! \!}     0    &  2n_2D_2  {\! \! \!}
  \end{array}{\! \! \! \! \!}
  \begin{array}\lgroup{c}\rgroup
  {\! \! \!}  v_{11}  {\! \! \!} \\
  {\! \! \!}  v_{12}  {\! \! \!}
  \end{array} \!
  +i\mu \!
  \begin{array}\lgroup{cc}\rgroup
  {\! \! \!}   0   &   n_2  {\! \! \!} \\
  {\! \! \!}   0   &    0   {\! \! \!}
  \end{array} {\! \! \! \! \!}
  \begin{array}\lgroup{c}\rgroup
  {\! \! \!}  1  {\! \! \!} \\
  {\! \! \!}  0  {\! \! \!}
  \end{array}= \mu
  \begin{array}\lgroup{c}\rgroup
  {\! \! \!}   n_2D_2v_{11}  {\! \! \!} \\
  {\! \! \!}  2n_2D_2v_{12}  {\! \! \!}
  \end{array}_0 \!=
  \begin{array}\lgroup{c}\rgroup
  {\! \! \!}  0  {\! \! \!} \\
  {\! \! \!}  0  {\! \! \!}
  \end{array}.
\]\\
  From this system of equations, we obtain that the associated vector $v_1$ has the form:
\[
  v_1=(1,0).
\]

  Let $v_2$~--- be the next associated eigenvector, then $v_2$ can be defined
  by the system of equations:
\[
  \mathcal{C}_0(0)v_2+\frac{\partial}{\partial\alpha}\mathcal{C}_0(0)v_1+
  \frac{1}{2}\frac{\partial^2}{\partial\alpha^2}\mathcal{C}_0(0)v_0=
\]
\smallskip
\[
  =-\mu \!
  \begin{array}\lgroup{cc}\rgroup
  {\! \! \!}  D^2_2   &    0     {\! \! \!}  \\
  {\! \! \!}     0    &  2D^2_2  {\! \! \!}
  \end{array}{\! \! \! \! \!}
  \begin{array}\lgroup{c}\rgroup
  {\! \! \!}  v_{21}  {\! \! \!} \\
  {\! \! \!}  v_{22}  {\! \! \!}
  \end{array} \!
  -i\mu \!
  \begin{array}\lgroup{cc}\rgroup
  {\! \! \!}   0   &   D_2  {\! \! \!} \\
  {\! \! \!}  D_2  &    0   {\! \! \!}
  \end{array} {\! \! \! \! \!}
  \begin{array}\lgroup{c}\rgroup
  {\! \! \!}  0  {\! \! \!} \\
  {\! \! \!}  1  {\! \! \!}
  \end{array} \! +\mu
  \begin{array}\lgroup{cc}\rgroup
  {\! \! \!}   2   &   0  {\! \! \!} \\
  {\! \! \!}   0   &   1  {\! \! \!}
  \end{array} {\! \! \! \! \!}
  \begin{array}\lgroup{c}\rgroup
  {\! \! \!}  1  {\! \! \!} \\
  {\! \! \!}  0  {\! \! \!}
  \end{array} =
\]
\smallskip
\[
  =-\mu
  \begin{array}\lgroup{c}\rgroup
  {\! \! \!}   D^2_2v_{21}-2  {\! \! \!} \\
  {\! \! \!}  2D^2_2v_{22}    {\! \! \!}
  \end{array} \!=
  \begin{array}\lgroup{c}\rgroup
  {\! \! \!}  0  {\! \! \!} \\
  {\! \! \!}  0  {\! \! \!}
  \end{array};
\]

\bigskip
\[
  \left.\mathcal{U}(0)v_2+\frac{\partial}{\partial\alpha}\mathcal{U}(0)v_1\right|^1_0=
\]
\smallskip
\[
  =\mu \!
  \begin{array}\lgroup{cc}\rgroup
  {\! \! \!}  n_2D_2   &   0     {\! \! \!}  \\
  {\! \! \!}     0    &  2n_2D_2  {\! \! \!}
  \end{array}{\! \! \! \! \!}
  \begin{array}\lgroup{c}\rgroup
  {\! \! \!}  v_{21}  {\! \! \!} \\
  {\! \! \!}  v_{22}  {\! \! \!}
  \end{array} \!
  +i\mu \!
  \begin{array}\lgroup{cc}\rgroup
  {\! \! \!}   0   &   n_2  {\! \! \!} \\
  {\! \! \!}   0   &    0   {\! \! \!}
  \end{array} {\! \! \! \! \!}
  \begin{array}\lgroup{c}\rgroup
  {\! \! \!}  0  {\! \! \!} \\
  {\! \! \!}  1  {\! \! \!}
  \end{array}= \mu
  \begin{array}\lgroup{c}\rgroup
  {\! \! \!}   n_2D_2v_{21}+in_2  {\! \! \!} \\
  {\! \! \!}   2n_2D_2v_{22}      {\! \! \!}
  \end{array}_0 \!=
  \begin{array}\lgroup{c}\rgroup
  {\! \! \!}  0  {\! \! \!} \\
  {\! \! \!}  0  {\! \! \!}
  \end{array}.
\medskip
\]
  The system of equation have solution $v_2$ if the next boundary problem
  is solvable:
\begin{align} \label{n48}
      \left \{
        \begin{matrix}
          \quad\mu D^2_2v_{21} & = & -2\mu
          \\
          \left.{\mu n_2D_2v_{21}}\right|^1_0 & = & -i\mu n_2.
        \end{matrix}
       \right.
\end{align}
  Therefore, the boundary problem (\ref{n48}) is solvable if and
  only if we have the next equality:
\[
  -2\mu\!\int\limits^1_0\!dx+(i\mu n_2,1)_0=0.
\]
  From the relation (\ref{n33}), we obtain that
\[
  (i\mu n_2,1)_0=\mu(in_2,1)_0=\mu(iD_21,1)-\mu(1,iD_21)=0.
\]
  From here we have
\[
    -2\mu\!\int\limits^1_0\!dx+(i\mu n_2,1)_0=-2\mu\!\int\limits^1_0\!dx=-2\mu<0.
\]
  For this inequality the boundary problem (\ref{n48})
  and at the same time the system of equations for the
  vector are not solvable. Consequently, the chain of the vectors:
\[
  v_0=(1,0) \;\: \emph{and} \;\; v_1=(0,1)\text{~---} \, \emph{\,is maximal}\,.
\]
\\
\bigskip
  b) Let $u_1$~--- be an associated vector to be eigenvector $u_ 0$,
  then $u_1$ can be defined by the system of equations:
\[
  \mathcal{C}_0(0)u_1+\frac{\partial}{\partial\alpha}\mathcal{C}_0(0)u_0=
\]
\[
  =-\mu \!
  \begin{array}\lgroup{cc}\rgroup
  {\! \! \!}  D^2_2   &    0     {\! \! \!}  \\
  {\! \! \!}     0    &  2D^2_2  {\! \! \!}
  \end{array}{\! \! \! \! \!}
  \begin{array}\lgroup{c}\rgroup
  {\! \! \!}  u_{11}  {\! \! \!} \\
  {\! \! \!}  u_{12}  {\! \! \!}
  \end{array} \!
  -i\mu \!
  \begin{array}\lgroup{cc}\rgroup
  {\! \! \!}   0   &   D_2  {\! \! \!} \\
  {\! \! \!}  D_2  &    0   {\! \! \!}
  \end{array} {\! \! \! \! \!}
  \begin{array}\lgroup{c}\rgroup
  {\! \! \!}  0  {\! \! \!} \\
  {\! \! \!}  1  {\! \! \!}
  \end{array}= -\mu
  \begin{array}\lgroup{c}\rgroup
  {\! \! \!}   D^2_2u_{11}  {\! \! \!} \\
  {\! \! \!}  2D^2_2u_{12}  {\! \! \!}
  \end{array} \!=
  \begin{array}\lgroup{c}\rgroup
  {\! \! \!}  0  {\! \! \!} \\
  {\! \! \!}  0  {\! \! \!}
  \end{array};
\]

\bigskip
\[
  \left.\mathcal{U}(0)u_1+\frac{\partial}{\partial\alpha}\mathcal{U}(0)u_0\right|^1_0=
\]
\smallskip
\[
  =\mu \!
  \begin{array}\lgroup{cc}\rgroup
  {\! \! \!}  n_2D_2   &   0     {\! \! \!}  \\
  {\! \! \!}     0    &  2n_2D_2  {\! \! \!}
  \end{array}{\! \! \! \! \!}
  \begin{array}\lgroup{c}\rgroup
  {\! \! \!}  u_{11}  {\! \! \!} \\
  {\! \! \!}  u_{12}  {\! \! \!}
  \end{array} \!
  +i\mu \!
  \begin{array}\lgroup{cc}\rgroup
  {\! \! \!}   0   &   n_2  {\! \! \!} \\
  {\! \! \!}   0   &    0   {\! \! \!}
  \end{array} {\! \! \! \! \!}
  \begin{array}\lgroup{c}\rgroup
  {\! \! \!}  0  {\! \! \!} \\
  {\! \! \!}  1  {\! \! \!}
  \end{array}= \mu
  \begin{array}\lgroup{c}\rgroup
  {\! \! \!}   n_2D_2u_{11} +in_2 {\! \! \!} \\
  {\! \! \!}  2n_2D_2u_{12}  {\! \! \!}
  \end{array}_0 \!=
  \begin{array}\lgroup{c}\rgroup
  {\! \! \!}  0  {\! \! \!} \\
  {\! \! \!}  0  {\! \! \!}
  \end{array}.
\]\\
  From this system of equations, we obtain that the associated vector $u_1$
  has the form:
\[
  \begin{matrix}
  u_1=-i(x_2-\frac{1}{2},0).
  \end{matrix}
\]

  Let $u_2$~--- be the next associated eigenvector, then $u_2$ can be defined
  by the system of equations:
\[
  \mathcal{C}_0(0)u_2+\frac{\partial}{\partial\alpha}\mathcal{C}_0(0)u_1+
  \frac{1}{2}\frac{\partial^2}{\partial\alpha^2}\mathcal{C}_0(0)u_0=
\]
\[
  =-\mu \!
  \begin{array}\lgroup{cc}\rgroup
  {\! \! \!}  D^2_2   &    0     {\! \! \!}  \\
  {\! \! \!}     0    &  2D^2_2  {\! \! \!}
  \end{array}{\! \! \! \! \!}
  \begin{array}\lgroup{c}\rgroup
  {\! \! \!}  u_{21}  {\! \! \!} \\
  {\! \! \!}  u_{22}  {\! \! \!}
  \end{array} \!
  -i\mu \!
  \begin{array}\lgroup{cc}\rgroup
  {\! \! \!}   0   &   D_2  {\! \! \!} \\
  {\! \! \!}  D_2  &    0   {\! \! \!}
  \end{array} {\! \! \! \! \!}
  \begin{array}\lgroup{c}\rgroup
  {\! \! \! \!}  -i(x_2-\frac{1}{2})  {\! \! \!} \\
  {\! \! \!}            0  {\! \! \!}
  \end{array} \! +\mu
  \begin{array}\lgroup{cc}\rgroup
  {\! \! \!}   2   &   0  {\! \! \!} \\
  {\! \! \!}   0   &   1  {\! \! \!}
  \end{array} {\! \! \! \! \!}
  \begin{array}\lgroup{c}\rgroup
  {\! \! \!}  0  {\! \! \!} \\
  {\! \! \!}  1  {\! \! \!}
  \end{array} =
\]
\smallskip
\[
  =-\mu
  \begin{array}\lgroup{c}\rgroup
  {\! \! \!}   D^2_2u_{21}  {\! \! \!} \\
  {\! \! \!}  2D^2_2u_{22}    {\! \! \!}
  \end{array} \!=
  \begin{array}\lgroup{c}\rgroup
  {\! \! \!}  0  {\! \! \!} \\
  {\! \! \!}  0  {\! \! \!}
  \end{array};
\]

\bigskip
\[
  \left.\mathcal{U}(0)u_2+\frac{\partial}{\partial\alpha}\mathcal{U}(0)u_1\right|^1_0=
\]
\smallskip
\[
  =\mu \!
  \begin{array}\lgroup{cc}\rgroup
  {\! \! \!}  n_2D_2   &   0     {\! \! \!}  \\
  {\! \! \!}     0    &  2n_2D_2  {\! \! \!}
  \end{array}{\! \! \! \! \!}
  \begin{array}\lgroup{c}\rgroup
  {\! \! \!}  u_{21}  {\! \! \!} \\
  {\! \! \!}  u_{22}  {\! \! \!}
  \end{array} \!
  +i\mu \!
  \begin{array}\lgroup{cc}\rgroup
  {\! \! \!}   0   &   n_2  {\! \! \!} \\
  {\! \! \!}   0   &    0   {\! \! \!}
  \end{array} {\! \! \! \! \!}
  \begin{array}\lgroup{c}\rgroup
  {\! \! \! \!}  -i(x_2-\frac{1}{2})  {\! \! \!} \\
  {\! \! \!}             0           {\! \! \!}
  \end{array}= \mu
  \begin{array}\lgroup{c}\rgroup
  {\! \! \!}   n_2D_2u_{21}  {\! \! \!} \\
  {\! \! \!}   2n_2D_2u_{22}      {\! \! \!}
  \end{array}_0 \!=
  \begin{array}\lgroup{c}\rgroup
  {\! \! \!}  0  {\! \! \!} \\
  {\! \! \!}  0  {\! \! \!}
  \end{array}.
\]\\
  From this system of equations, we obtain that the associated vector
  $u_2$ has the form:
\[
  u_2=(1,0).
\]

  Let $u_3$~--- be the next associated eigenvector, then $u_3$ can be defined
  by the system of equations:
\[
  \mathcal{C}_0(0)u_3+\frac{\partial}{\partial\alpha}\mathcal{C}_0(0)u_2+
  \frac{1}{2}\frac{\partial^2}{\partial\alpha^2}\mathcal{C}_0(0)u_1=
\]
\[
  =-\mu \!
  \begin{array}\lgroup{cc}\rgroup
  {\! \! \!}  D^2_2   &    0     {\! \! \!}  \\
  {\! \! \!}     0    &  2D^2_2  {\! \! \!}
  \end{array}{\! \! \! \! \!}
  \begin{array}\lgroup{c}\rgroup
  {\! \! \!}  u_{31}  {\! \! \!} \\
  {\! \! \!}  u_{32}  {\! \! \!}
  \end{array} \!
  -i\mu \!
  \begin{array}\lgroup{cc}\rgroup
  {\! \! \!}   0    &   D_2  {\! \! \!} \\
  {\! \! \!}  D_2  &    0   {\! \! \!}
  \end{array} {\! \! \! \! \!}
  \begin{array}\lgroup{c}\rgroup
  {\! \! \!}  1  {\! \! \!} \\
  {\! \! \!}  0  {\! \! \!}
  \end{array} \! +\mu
  \begin{array}\lgroup{cc}\rgroup
  {\! \! \!}   2   &   0  {\! \! \!} \\
  {\! \! \!}   0   &   1  {\! \! \!}
  \end{array} {\! \! \! \! \!}
  \begin{array}\lgroup{c}\rgroup
  {\! \! \! \!}  -i(x_2-\frac{1}{2})  {\! \! \!} \\
  {\! \! \!}           0  {\! \! \!}
  \end{array} =
\]
\smallskip
\[
  =-\mu
  \begin{array}\lgroup{c}\rgroup
  {\! \! \!}   D^2_2u_{31}+2i(x_2-\frac{1}{2})  {\! \! \!} \\
  {\! \! \!}         2D^2_2u_{32}    {\! \! \!}
  \end{array} \!=
  \begin{array}\lgroup{c}\rgroup
  {\! \! \!}  0  {\! \! \!} \\
  {\! \! \!}  0  {\! \! \!}
  \end{array};
\]

\bigskip
\[
  \left.\mathcal{U}(0)u_3+\frac{\partial}{\partial\alpha}\mathcal{U}(0)u_2\right|^1_0=
\]
\smallskip
\[
  =\mu \!
  \begin{array}\lgroup{cc}\rgroup
  {\! \! \!}  n_2D_2   &   0     {\! \! \!}  \\
  {\! \! \!}     0    &  2n_2D_2  {\! \! \!}
  \end{array}{\! \! \! \! \!}
  \begin{array}\lgroup{c}\rgroup
  {\! \! \!}  u_{31}  {\! \! \!} \\
  {\! \! \!}  u_{32}  {\! \! \!}
  \end{array} \!
  +i\mu \!
  \begin{array}\lgroup{cc}\rgroup
  {\! \! \!}   0   &   n_2  {\! \! \!} \\
  {\! \! \!}   0   &    0   {\! \! \!}
  \end{array} {\! \! \! \! \!}
  \begin{array}\lgroup{c}\rgroup
  {\! \! \!}  1  {\! \! \!} \\
  {\! \! \!}  0  {\! \! \!}
  \end{array}= \mu
  \begin{array}\lgroup{c}\rgroup
  {\! \! \!}   n_2D_2u_{31}  {\! \! \!} \\
  {\! \! \!}   2n_2D_2u_{32}      {\! \! \!}
  \end{array}_0 \!=
  \begin{array}\lgroup{c}\rgroup
  {\! \! \!}  0  {\! \! \!} \\
  {\! \! \!}  0  {\! \! \!}
  \end{array}.
\]\\
  From this system of equations, we examine the next boundary problem:
\begin{align} \label{n49}
      \left \{
        \begin{matrix}
          \quad\mu D^2_2u_{31} & = & -2i\mu(x_2-\frac{1}{2})
          \\
          \left.{\mu n_2D_2u_{31}}\right|^1_0 & = & 0 .
        \end{matrix}
       \right.
\end{align}
  The boundary problem (\ref{n49}) has solution if we have
  the next equality is true:
\[
  -2i\mu\!\int\limits^1_0\begin{matrix}(x_2-\frac{1}{2})\,dx=0.\end{matrix}
\]
  Computing this relation we find
\[
  -2i\mu\!\int \limits^1_0\left.\begin{matrix}(x_2-\frac{1}{2})\,dx=
  -2i\mu(\frac{1}{2}x^2_2-\frac{1}{2}x_2)\end{matrix}\right|^1_0=0.
\]
  From here the boundary problem (\ref{n49}) is solvable and the
  vector $u_3$ has the form:
\[
  \begin{matrix}
    u_3=-i(\frac{1}{3}x^3_2-\frac{1}{2}x^2_2,0).
  \end{matrix}
\]

\bigskip
  Let $u_4$~--- be the next associated eigenvector, then $u_4$ can be defined
  by the system of equations:
\[
  \mathcal{C}_0(0)u_4+\frac{\partial}{\partial\alpha}\mathcal{C}_0(0)u_3+
  \frac{1}{2}\frac{\partial^2}{\partial\alpha^2}\mathcal{C}_0(0)u_2=
\smallskip
\]
\[
  =-\mu \!
  \begin{array}\lgroup{cc}\rgroup
  {\! \! \!}  D^2_2   &    0     {\! \! \!}  \\
  {\! \! \!}     0    &  2D^2_2  {\! \! \!}
  \end{array}{\! \! \! \! \!}
  \begin{array}\lgroup{c}\rgroup
  {\! \! \!}  u_{41}  {\! \! \!} \\
  {\! \! \!}  u_{42}  {\! \! \!}
  \end{array} \!
  -i\mu \!
  \begin{array}\lgroup{cc}\rgroup
  {\! \! \!}   0    &   D_2  {\! \! \!} \\
  {\! \! \!}  D_2  &    0   {\! \! \!}
  \end{array} {\! \! \! \! \!}
  \begin{array}\lgroup{c}\rgroup
  {\! \! \! \!} -i(\frac{1}{3}x^3_2-\frac{1}{2}x^2_2)  {\! \! \!} \\
  {\! \! \!}                        0                   {\! \! \!}
  \end{array} \! +\mu
  \begin{array}\lgroup{cc}\rgroup
  {\! \! \!}   2   &   0  {\! \! \!} \\
  {\! \! \!}   0   &   1  {\! \! \!}
  \end{array} {\! \! \! \! \!}
  \begin{array}\lgroup{c}\rgroup
  {\! \! \!}  1  {\! \! \!} \\
  {\! \! \!}  0  {\! \! \!}
  \end{array} =
\]
\smallskip
\[
  =-\mu
  \begin{array}\lgroup{c}\rgroup
  {\! \! \!}   D^2_2u_{41}-2  {\! \! \!} \\
  {\! \! \! \!}   2D^2_2u_{42}+x^2_2-x_2  {\! \! \! \!}
  \end{array} \!=
  \begin{array}\lgroup{c}\rgroup
  {\! \! \!}  0  {\! \! \!} \\
  {\! \! \!}  0  {\! \! \!}
  \end{array};
\]

\bigskip
\[
  \left.\mathcal{U}(0)u_4+\frac{\partial}{\partial\alpha}\mathcal{U}(0)u_3\right|^1_0=
\]
\smallskip
\[
  =\mu \!
  \begin{array}\lgroup{cc}\rgroup
  {\! \! \!}  n_2D_2   &   0     {\! \! \!}  \\
  {\! \! \!}     0    &  2n_2D_2  {\! \! \!}
  \end{array}{\! \! \! \! \!}
  \begin{array}\lgroup{c}\rgroup
  {\! \! \!}  u_{41}  {\! \! \!} \\
  {\! \! \!}  u_{42}  {\! \! \!}
  \end{array} \!
  +i\mu \!
  \begin{array}\lgroup{cc}\rgroup
  {\! \! \!}   0   &   n_2  {\! \! \!} \\
  {\! \! \!}   0   &    0   {\! \! \!}
  \end{array} {\! \! \! \! \!}
  \begin{array}\lgroup{c}\rgroup
  {\! \! \! \!} -i(\frac{1}{3}x^3_2-\frac{1}{2}x^2_2)   {\! \! \!} \\
  {\! \! \!}                     0                      {\! \! \!}
  \end{array}= \mu
  \begin{array}\lgroup{c}\rgroup
  {\! \! \!}   n_2D_2u_{41}  {\! \! \!} \\
  {\! \! \!}   2n_2D_2u_{42}      {\! \! \!}
  \end{array}_0 \!=
  \begin{array}\lgroup{c}\rgroup
  {\! \! \!}  0  {\! \! \!} \\
  {\! \! \!}  0  {\! \! \!}
  \end{array}.
\]\\
  From this system of equations, we examine the next boundary problem:
\begin{align} \label{n50}
      \left \{
        \begin{matrix}
          \quad 2\mu D^2_2u_{42} & = & -\mu(x^2_2-x_2)
          \\
          \left.{2\mu n_2D_2u_{42}}\right|^1_0 & = & 0 .
        \end{matrix}
       \right.
\end{align}
  The boundary problem (\ref{n50}) is solvable if we have the equality:
\[
  \mu\!\int\limits^1_0(x^2_2-x_2)\,dx=0.
\]
  We examine the value of the integral
\[
  \mu\!\int \limits^1_0\left.\begin{matrix}(x^2_2-x_2)\,dx=
  \mu(\frac{1}{3}x^3_2-\frac{1}{2}x^2_2)\end{matrix}\right|^1_0=-\frac{\mu}{6}<0.
\]
  For this inequality the boundary problem (\ref{n50})
  and at the same time the system of equations for the
  vector are not solvable. Consequently, the chain of the vectors:
\[
  \begin{matrix}
    u_1=(0,1), \, u_2=-i(x_2-\frac{1}{2}), \, u_2=(1,0) \; \, \emph{and} \;
    \; u_3=-i(\frac{1}{3}x^3_2-\frac{1}{2}x^2_2,0) \text{~---} \,
    \emph{\,is maximal}\,.
  \end{matrix}
\]

  The theorem is proved.

\bigskip
  From theorems~(\ref{T2}) and~(\ref{T6}) For the isotropic structure
  in the plane, when the Lame constants are $\mu>0$ and $\lambda=0$,
  and for the static case, i.e., when $\omega=0$, we obtain that if
  to all the eigen-and associated vector of the spectral problem
  $\mathcal{L}_0(\alpha)$, corresponding to the eigenvalues $\alpha_n$
  from upper semi-plane, adjoins the vectors $v_0, \: u_0  \; \text{and} \; \, u_1$
  from (\ref{n47}), then we obtain a complete and minimal system in
  the space $W^1_2[\,0,1] \; (\text{and in} \: L_2[\,0,1])$.\\

  Now we examine the second case for the isotropic structure,
  when the Lame constants are $\mu>0$ and $\lambda=\infty$, see \cite{SIS83,ALP,WHM40}.
  Then the Poisson's ratio has the value $\nu=\frac{1}{2}$, in this
  case the cubical compression (dilatation) is:
\[
  \varepsilon_{\text{vol}}=\text{div}\,u=\frac{1-2\nu}{E}\Theta,
\]
  where $\Theta=\sigma_x+\sigma_y+\sigma_z$ is the bulk stress,
  $\sigma_x, \, \sigma_y, \, \sigma_z$ are principal stresses and $E$
  is the Young's modulus (modulus of elasticity).
  At the same time we have
\[
  \varepsilon_{\text{vol}}=\text{div}\,u=e_{11}+e_{22}+e_{33}=0.
\]
  These formulas show
  that for the isotropic structure, when the Lame constants are
  $\mu>0$ and $\lambda=\infty$, the free energy $W$ of the system
  has the form:
\[
  W=\mu(e^2_{11}+e^2_{22}+e^2_{33}+2e^2_{12}+2e^2_{13}+2e^2_{23}).
\medskip
\]
  Therefore, with the corresponding changes, all the previous results
  are preserved for this case.\\

\bigskip

  In conclusion the author expresses his sincere thanks to A. G.
  Kostyuchenko for valuable remarks, I also thank the UACM for the support provided during the sabbatical period.\\

\end{document}